\documentclass[12pt]{amsart}%
\usepackage{amscd}
\usepackage{amsfonts}
\usepackage{graphicx}
\usepackage{amsmath}
\usepackage{amssymb}
\usepackage[margin=1.2 in]{geometry}%
\providecommand{\U}[1]{\protect\rule{.1in}{.1in}}
\providecommand{\U}[1]{\protect\rule{.1in}{.1in}}
\providecommand{\U}[1]{\protect\rule{.1in}{.1in}}
\newtheorem{theorem}{Theorem}
\newtheorem{acknowledgement}[theorem]{Acknowledgement}

\newtheorem{condition}[theorem]{Condition}

\newtheorem{definition}[theorem]{Definition}
\newtheorem{example}[theorem]{Example}

\newtheorem{lemma}[theorem]{Lemma}

\newtheorem{proposition}[theorem]{Proposition}
\newtheorem{remark}[theorem]{Remark}

\begin{document}
\date{\today}
\title[Sampling and Interpolation]{Sampling and Interpolation on Some Nilpotent Lie Groups}
\author[V. Oussa]{Vignon Oussa}
\address{Dept.\ of Mathematics \& Computer Science\\
Bridgewater State University\\
Bridgewater, MA 02325 U.S.A.\\
 }
\email{vignon.oussa@bridgew.edu}
\keywords{sampling, interpolation, nilpotent Lie groups, representations}
\subjclass[2000]{22E25, 22E27}

\begin{abstract}
Let $N$ be a non-commutative, simply connected, connected, two-step nilpotent
Lie group with Lie algebra $\mathfrak{n}$ such that $\mathfrak{n=a\oplus
b\oplus z}$, $\left[  \mathfrak{a},\mathfrak{b}\right]  \subseteq
\mathfrak{z},$ the algebras $\mathfrak{a},\mathfrak{b,z}$ are abelian,
$\mathfrak{a}=\mathbb{R}\text{-span}\left\{  X_{1},X_{2},\cdots,X_{d}\right\}
,$ and $\mathfrak{b}=\mathbb{R}\text{-span}\left\{  Y_{1},Y_{2},\cdots
,Y_{d}\right\}  .$ Also, we assume that $\det\left[  \left[  X_{i}%
,Y_{j}\right]  \right]  _{1\leq i,j\leq d}$ is a non-vanishing homogeneous
polynomial in the unknowns $Z_{1},\cdots,Z_{n-2d}$ where $\left\{
Z_{1},\cdots,Z_{n-2d}\right\}  $ is a basis for the center of the Lie algebra.
Using well-known facts from time-frequency analysis, we provide some precise
sufficient conditions for the existence of sampling spaces with the
interpolation property, with respect to some discrete subset of $N$. The
result obtained in this work can be seen as a direct application of
time-frequency analysis to the theory of nilpotent Lie groups. Several
explicit examples are computed. This work is a generalization of recent
results obtained for the Heisenberg group by Currey, and Mayeli in
\cite{Currey}.

\end{abstract}
\maketitle
\tableofcontents



\vskip 0.5cm


\vskip 0.5cm

\section{Introduction}

Let $N$ be a locally compact group, and let $\Gamma$ be a discrete subset of
$N.$ Let $\mathbf{H}$ be a left-invariant closed subspace of $L^{2}\left(
N\right)  $ consisting of continuous functions. We call $\mathbf{H}$ a
\textbf{sampling space (}Section $2.6$ \cite{Fuhr cont}) with respect to
$\Gamma$ (or $\Gamma$-sampling space) if the following properties hold. First,
the restriction mapping $R_{\Gamma}:\mathbf{H}\rightarrow l^{2}\left(
\Gamma\right)  ,$ $R_{\Gamma}f=\left(  f\left(  \gamma\right)  \right)
_{\gamma\in\Gamma}$ is an isometry. Secondly, there exists a vector
$S\in\mathbf{H}$ such that for any vector $f\in\mathbf{H},$ we have the
following expansion
\[
f\left(  x\right)  =\sum_{\gamma\in\Gamma}f\left(  \gamma\right)  S\left(
\gamma^{-1}x\right)
\]
with convergence in the norm of $\mathbf{H}$. The vector $S$ is called a
\textbf{sinc-type} function, and if $R_{\Gamma}$ is surjective, we say that
the sampling space $\mathbf{H}$ has the \textbf{interpolation property}.

The simplest example of a sampling space with interpolation property over a
nilpotent Lie group is provided by the well-known Whittaker, Shannon,
Kotel'nikov Theorem (see Example $2.52$ \cite{Fuhr cont}) which we recall
here. Let $C\left(  \mathbb{R}\right)  $ be the vector space of complex-valued
continous functions on the real line, and let
\[
\mathbf{H=}\left\{  f\in L^{2}\left(  \mathbb{R}\right)  \cap C\left(
\mathbb{R}\right)  :\text{support }\widehat{f}\subseteq\left[
-0.5,0.5\right]  \right\}
\]
where $f\mapsto\widehat{f}$ is the Fourier transform of $f$ and is defined as
$\widehat{f}\left(  \xi\right)  =\int_{\mathbb{R}}f\left(  x\right)  e^{-2\pi
ix\xi}dx$ whenever $f\in L^{1}\left(  \mathbb{R}\right)  .$ Then $\mathbf{H}$
is a sampling space which has the interpolation property with associated
sinc-type function
\[
S\left(  x\right)  =\left\{
\begin{array}
[c]{c}%
\dfrac{\sin\pi x}{\pi x}\text{ if }x\neq0\\
1\text{ if }x=0
\end{array}
\right.  .
\]

To the best of our knowledge, the first example of a sampling space with
interpolation property on a non-commutative nilpotent Lie group, using the
Plancherel transform was defined over the three-dimensional Heisenberg Lie
group. This example is due to a remarkable result of Currey and Mayeli
\cite{Currey}. The specific definition of bandlimited spaces by the Plancherel
transform used in \cite{Currey}, was taken from \cite{Fuhr cont}, Chapter $6,$
where a very precise characterization of sampling spaces over the Heisenberg
group was provided. Moreover, sampling spaces using a similar definition of
bandlimitation were studied in \cite{oussa} and \cite{oussa1} for a class of
nilpotent Lie groups which contains the Heisenberg Lie groups. This class of
groups was first introduced by the author in \cite{oussa}. However, nothing
was said about the interpolation property of the sampling spaces described in
\cite{oussa}. In fact, the question of existence of sampling spaces with
interpolation property on some non-commutative nilpotent Lie groups is a
challenging problem which is the central focus of this paper.

Let $N$ be a simply connected, connected, two-step nilpotent Lie group with
Lie algebra $\mathfrak{n}$ of dimension $n$ satisfying the following conditions.

\begin{condition}
\label{conditiion} Assume that $\mathfrak{n=a\oplus b\oplus z}$, where
$\left[  \mathfrak{a},\mathfrak{b}\right]  \subseteq\mathfrak{z,}$
$\mathfrak{a},\mathfrak{b,z}$ are abelian algebras such that
\begin{align*}
\mathfrak{a}  &  =%
\mathbb{R}
\text{-span}\left\{  X_{1},X_{2},\cdots,X_{d}\right\}  ,\\
\mathfrak{b}  &  =%
\mathbb{R}
\text{-span}\left\{  Y_{1},Y_{2},\cdots,Y_{d}\right\}  ,\\
\mathfrak{z}  &  =%
\mathbb{R}
\text{-span}\left\{  Z_{1},Z_{2},\cdots,Z_{n-2d}\right\}  ,
\end{align*}
$d\geq1,$ $n>2d$ and
\begin{equation}
\det\left[
\begin{array}
[c]{cccc}%
\left[  X_{1},Y_{1}\right]  & \left[  X_{1},Y_{2}\right]  & \cdots & \left[
X_{1},Y_{d}\right] \\
\left[  X_{2},Y_{1}\right]  & \left[  X_{2},Y_{2}\right]  & \cdots & \left[
X_{2},Y_{d}\right] \\
\vdots & \vdots & \cdots & \vdots\\
\left[  X_{d},Y_{1}\right]  & \left[  X_{d},Y_{2}\right]  & \cdots & \left[
X_{d},Y_{d}\right]
\end{array}
\right]  \label{matrix}%
\end{equation}
is a non-vanishing homogeneous polynomial in the unknowns $Z_{1}%
,\cdots,Z_{n-2d}.$
\end{condition}

We remark that the entries of the matrix $\left[  \left[  X_{i},Y_{j}\right]
\right]  _{1\leq i,j\leq d}$ are linear combinations of a basis of the
commutator ideal of $\left[  \mathfrak{n},\mathfrak{n}\right]  $ which can be
taken to be a subset of $\left\{  Z_{1},Z_{2},\cdots,Z_{n-2d}\right\}  .$ The
object described in (\ref{matrix}) is then obtained by formally computing the
determinant in the unknowns $Z_{1},\cdots,Z_{n-2d}.$ Also, given a Lie algebra
$\mathfrak{n}$ which satisfies all assumptions in Condition \ref{conditiion},
it is worth mentioning that since we require $n-2d$ to be positive, then
$\dim\mathfrak{z}=n-2d\geq1$ and $\mathfrak{n}$ must necessarily be non-abelian.

One very appealing fact about these groups is the following. The
infinite-dimensional irreducible representations of any group satisfying the
conditions given above are related to the well-known Schr\"{o}dinger
representations \cite{oussa, oussa1}. Thus, the advantage of working with this
class of groups is that, we are able to exploit well-known theorems from
time-frequency analysis.

Let $N$ be a nilpotent Lie group satisfying Condition \ref{conditiion}. We
deal with the existence of left-invariant subspaces of $L^{2}\left(  N\right)
$ which are sampling spaces which have the interpolation property. More
precisely, we investigate conditions under which sampling provides an
orthonormal basis which is generated by shifting a single function. The work
presented here provides a natural generalization of recent results obtained
for the Heisenberg group in \cite{Currey}. We offer precise and explicit
sufficient conditions for sampling spaces, which also have the interpolation
property with respect to some discrete set $\Gamma\subset N$.

We organize this paper as follows. The second section deals with some
preliminary results which can be found in \cite{oussa,oussa1,Corwin}. In the
third section, we introduce a natural notion of bandlimitation for the class
of groups considered, and we state the main results (Theorem \ref{main} and
Theorem \ref{steps}) of the paper. In the fourth section, we prove results
related to sampling and frames for the class of groups considered here. The
results obtained in the fourth section are crucial for the proofs of Theorem
\ref{main} and Theorem \ref{steps} which are provided in the last section.
Finally, explicit examples are computed.

\section{Preliminaries}

Let us start by setting up some notation. In this paper, all representations
are strongly continuous and unitary, unless we state otherwise. All sets are
measurable, and given two equivalent unitary representations $\tau$ and $\pi,$
we write $\tau\cong\pi.$ We also use the same notation for isomorphic Hilbert
spaces. The characteristic function of a set $\mathbf{E}$ is written as
$\chi_{\mathbf{E}},$ and the cardinal number of a set $\mathbf{I}$ is denoted
by $\mathrm{card}\left(  \mathbf{I}\right)  .$ $V^{\ast}$ stands for the dual
vector space of a vector space of $V.$ Let $v$ be a vector in $%
\mathbb{R}
^{n}$. $v^{tr}$ stands for the transpose of the vector $v.$ The Fourier
transform of a suitable function $f$ defined over a commutative domain is
written as $\widehat{f},$ and the conjugate of a complex number $z$ is denoted
$\overline{z}.$ The general linear group of $\mathbb{R}^{n}$ is denoted
$GL_{n}\left(  \mathbb{R}\right)  .$ Let $v,w$ be two vectors in some Hilbert
space. We write $v\perp w$ to denote that the vectors are orthogonal to each
other with respect to the inner product which the given Hilbert space is
endowed with.

Now, we will provide a short introduction to the theory of direct integrals
which is also nicely exposed in Section $3.3$ of \cite{Fuhr cont}. Let
$\left\{  H_{\alpha}\right\}  _{\alpha\in A}$ be a family of separable Hilbert
spaces indexed by a set $A.$ Let $\mu$ be a measure defined in $A.$ We define
the direct integral of this family of Hilbert spaces with respect to $\mu$ as
the space which consists of functions $f$ defined on the parameter space $A$
such that $f\left(  \alpha\right)  $ is an element of $H_{\alpha}$ for each
$\alpha\in A$, and
\[
\int_{A}\left\Vert f\left(  \alpha\right)  \right\Vert _{H_{\alpha}}^{2}%
d\mu\left(  \alpha\right)  <\infty
\]
with some additional measurability conditions which we will clarify. A family
of separable Hilbert spaces $\left\{  H_{\alpha}\right\}  _{\alpha\in A}$
indexed by a Borel set $A$ is called a field of Hilbert spaces over $A.$ Next,
a map
\[
f:A\rightarrow%
{\displaystyle\bigcup\limits_{\alpha\in A}}
H_{\alpha}\text{ such that }f\left(  \alpha\right)  \in H_{\alpha}%
\]
is called a vector field on $A.$ A measurable field of Hilbert spaces over the
indexing set $A$ is a field of Hilbert spaces $\left\{  H_{\alpha}\right\}
_{\alpha\in A}$ together with a countable set $\left\{  e_{j}\right\}  _{j}$
of vector fields such that

\begin{enumerate}
\item the functions $\alpha\mapsto\left\langle e_{j}\left(  \alpha\right)
,e_{k}\left(  \alpha\right)  \right\rangle _{H_{\alpha}}$ are measurable for
all $j,k,$

\item the linear span of $\left\{  e_{k}\left(  \alpha\right)  \right\}  _{k}$
is dense in $H_{\alpha}$ for each $\alpha\in A.$
\end{enumerate}

The direct integral of the spaces $H_{\alpha}$ with respect to the measure
$\mu$ is denoted by
\[
\int_{A}^{\oplus}H_{\alpha}d\mu\left(  \alpha\right)
\]
and is the space of measurable vector fields $f$ on $A$ such that
\[
\int_{A}\left\Vert f\left(  \alpha\right)  \right\Vert _{H_{\alpha}}^{2}%
d\mu\left(  \alpha\right)  <\infty.
\]
The inner product for this Hilbert space is naturally obtained as follows. For
$f,g\in\int_{A}^{\oplus}H_{\alpha}d\mu\left(  \alpha\right)  ,$%
\[
\left\langle f,g\right\rangle =\int_{A}\left\langle f\left(  \alpha\right)
,g\left(  \alpha\right)  \right\rangle _{H_{\alpha}}d\mu\left(  \alpha\right)
.
\]
This theory of direct integrals will play an important role in the definition
of bandlimited spaces in our work.

Let $N$ be a non-commutative connected, simply connected nilpotent Lie group
with Lie algebra $\mathfrak{n}$ over the reals with some additional
assumptions described in Condition \ref{conditiion}.

Notice that if $\mathfrak{n}$ is the three-dimensional Heisenberg Lie algebra
which is spanned by vectors $X,Y,Z$ such that $\left[  X,Y\right]  =Z,$ then
we may define
\[
\mathfrak{a=}\text{ }%
\mathbb{R}
X,\text{ }\mathfrak{b=}\text{ }%
\mathbb{R}
Y\text{ and }\mathfrak{z=}\text{ }%
\mathbb{R}
Z.
\]
Although this is trivial, we make the following observation: $\det\left[
X,Y\right]  =Z$ is a non-vanishing homogeneous polynomial in the variable $Z.$
Therefore, the class of groups satisfying the conditions described above
contains groups which can be seen as some generalization of the Heisenberg Lie
groups. Let
\[
\mathfrak{B}=\left\{  T_{1},T_{2},\cdots,T_{n}\right\}
\]
be a basis for the Lie algebra $\mathfrak{n}.$ We say that $\mathfrak{B}$ is a
\textbf{strong Malcev basis }(see Page $10$ \cite{Corwin}\textbf{)} if and
only if for each $1\leq j\leq n$ the real span of
\[
\left\{  T_{1},T_{2},\cdots,T_{j}\right\}
\]
is an ideal of $\mathfrak{n.}$ For the class of groups considered in this
work, in order to obtain a strong Malcev basis, it suffices to define
$Z_{k}=T_{k}\text{ if }1\leq k\leq n-2d.$ Next, if $n-2d+1\leq k\leq n-d,$
then $k=n-2d+j$ for some $j\in\left\{  1,2,\cdots,d\right\}  $ and
$T_{k}=Y_{j}.$ Finally, if $n-d+1\leq k\leq n$ then $k=n-d+j$ for
$j\in\left\{  1,2,\cdots,d\right\}  $ and in this case $T_{k}=X_{j}.$ Fixing
such a strong Malcev basis of the Lie algebra $\mathfrak{n}$, a typical
element of the Lie group $N$ is written as follows:
\[
\exp\left(  \sum_{k=1}^{n-2d}z_{k}Z_{k}\right)  \exp\left(  \sum_{k=1}%
^{d}y_{k}Y_{k}\right)  \exp\left(  \sum_{k=1}^{d}x_{k}X_{k}\right)  .
\]
The subgroup
\[
\exp\left(  \sum_{k=1}^{n-2d}\mathbb{R}Z_{k}\right)
\]
is the center of the Lie group $N$ and the subgroup
\[
\exp\left(  \sum_{k=1}^{n-2d}\mathbb{R}Z_{k}\right)  \exp\left(  \sum
_{k=1}^{d}\mathbb{R}Y_{k}\right)
\]
is a maximal normal abelian subgroup of $N.$ Moreover, $N$ is a step-two
nilpotent Lie group since the commutator ideal $\left[  \mathfrak{n,n}\right]
$ is central. Let us now collect some additional basic facts about groups
satisfying Condition \ref{conditiion}.

\begin{proposition}
\label{faithful}Let $N$ be a nilpotent Lie group satisfying the conditions
given above. There is a finite dimensional faithful representation of $N$ in
$GL\left(  n+1,%
\mathbb{R}
\right)  $ for $n\geq3.$
\end{proposition}

\begin{proof}
Clearly if $n<3,$ then $\mathfrak{n}$ must be abelian. Thus, we must assume
that $n\geq3.$ First, let $\mathfrak{n}_{1}=\mathfrak{a\oplus b}\oplus\left(
\mathfrak{z}\ominus\left[  \mathfrak{n,n}\right]  \right)  $ and
$\mathfrak{n}_{2}=\left[  \mathfrak{n,n}\right]  \subseteq\mathfrak{z}$ such
that $\mathfrak{n=n}_{1}\oplus\mathfrak{n}_{2}.$ Let $\alpha$ be a positive
real number. Next, we define an element $A_{\alpha}$ in the outer derivation
of $\mathfrak{n}$ acting by a diagonalizable action such that $\left[
A_{\alpha},U\right]  =\ln\left(  \alpha\right)  U$ for all $U\in
\mathfrak{n}_{1}$ and $\left[  A_{\alpha},Z\right]  =2\ln\left(
\alpha\right)  Z$ for all $Z\in\mathfrak{n}_{2}.$ Using the Jacobi identity,
it is fairly easy to see that indeed $A_{\alpha}$ defines a derivation. Next,
we consider the linear adjoint representation of $\mathfrak{g=n}\oplus%
\mathbb{R}
A_{\alpha},$ $\mathrm{ad}:\mathfrak{g}\rightarrow\mathfrak{gl}\left(
\mathfrak{g}\right)  $ and we define $G=\exp\left(  \mathrm{ad}\left(
\mathfrak{g}\right)  \right)  \ $which is a subgroup of $GL\left(
\mathfrak{g}\right)  .$ Fixing a strong Malcev basis for the Lie algebra
$\mathfrak{n,}$ the adjoint representation of $G$ acting on the vector space
$\mathfrak{g}$ is a faithful representation. Thus, $G=\exp\left(
\mathrm{ad}\left(  \mathfrak{g}\right)  \right)  $ is a Lie subgroup of
$GL\left(  \mathfrak{g}\right)  \cong GL\left(  n+1,%
\mathbb{R}
\right)  $. Since $N$ is isomorphic to $\exp\left(  \mathrm{ad}\left(
\mathfrak{n}\oplus\left\{  0\right\}  \right)  \right)  ,$ then $\exp\left(
\mathrm{ad}\left(  \mathfrak{n}\oplus\left\{  0\right\}  \right)  \right)  $
is an isomorphic copy of the Lie group $N$ inside $GL\left(  n+1,%
\mathbb{R}
\right)  .$
\end{proof}

Next, in order to make this paper self-contained, we will revisit the
Plancherel theory for the class of groups considered in this paper. We start
by fixing a strong Malcev basis for the Lie algebra of $\mathfrak{n}$. The
exponential function takes the Lebesgue measure on $\mathfrak{n}$ to a left
Haar measure on $N$ (see Theorem $1.2.10$ in \cite{Corwin}). Since $N$ is a
nilpotent Lie group, according to the orbit method (see \cite{Corwin}) all
irreducible representations of $N$ are parametrized by the coadjoint orbit of
$N$ in $\mathfrak{n}^{\ast},$ and it is possible to construct a smooth
cross-section $\Sigma$ in a Zariski open subset $\Omega\ $of $\mathfrak{n}%
^{\ast}$ which is dense and $N$-invariant such that $\Sigma$ meets every
coadjoint orbit in $\Omega$ at exactly one point. Let $\mathcal{P}$ be the
Plancherel transform on $L^{2}\left(  N\right)  $ and $\mathcal{F}$ the
Fourier transform defined on $L^{2}(N)\cap L^{1}(N)$ by
\[
\mathcal{F}\left(  f\right)  \left(  \lambda\right)  =\int_{N}f\left(
n\right)  \pi_{\lambda}\left(  n\right)  dn,
\]
where $\left\{  \pi_{\lambda}:\lambda\in\Sigma\right\}  $ parametrizes up to a
null set the unitary dual of $N.$ In fact, the set $\Sigma$ can be chosen such
that for each $\lambda\in\Sigma,$ the corresponding irreducible representation
$\pi_{\lambda}$ is realized as acting in the Hilbert space $L^{2}\left(
\mathbb{R}^{d}\right)  $ where $d$ is half of the dimension of the coadjoint
orbit of $\lambda.$ Next, it is well-known that
\[
\mathcal{P}:L^{2}\left(  N\right)  \rightarrow\int_{\Sigma}^{\oplus}%
L^{2}\left(  \mathbb{R}^{d}\right)  \otimes L^{2}\left(  \mathbb{R}%
^{d}\right)  d\mu\left(  \lambda\right)  ,
\]
such that the Plancherel transform is the extension of the Fourier transform
to $L^{2}(N)$ inducing the equality
\[
\left\Vert f\right\Vert _{L^{2}\left(  N\right)  }^{2}=\int_{\Sigma}\left\Vert
\mathcal{P}\left(  f\right)  \left(  \lambda\right)  \right\Vert
_{\mathcal{HS}}^{2}\text{ }d\mu\left(  \lambda\right)  .
\]
We recall that $||\cdot||_{\mathcal{HS}}$ denotes the Hilbert-Schmidt norm on
$L^{2}\left(  \mathbb{R}^{d}\right)  \otimes L^{2}\left(  \mathbb{R}%
^{d}\right)  $ and that the Hilbert space tensor product $L^{2}\left(
\mathbb{R}^{d}\right)  \otimes L^{2}\left(  \mathbb{R}^{d}\right)  $ is
defined as the space of bounded linear operators $T:L^{2}\left(
\mathbb{R}^{d}\right)  \rightarrow L^{2}\left(  \mathbb{R}^{d}\right)  $ such
that
\[
||T||_{\mathcal{HS}}=\sum_{k\in I}\left\Vert Te_{k}\right\Vert _{L^{2}\left(
\mathbb{R}^{d}\right)  }^{2}%
\]
where $\left(  e_{k}\right)  _{k\in I}$ is an orthonormal basis of
$L^{2}\left(  \mathbb{R}^{d}\right)  .$ Given arbitrary $S,T\in L^{2}\left(
\mathbb{R}^{d}\right)  \otimes L^{2}\left(  \mathbb{R}^{d}\right)  ,$ the
inner product of the operators $S,T$ is:%
\[
\left\langle S,T\right\rangle _{_{\mathcal{HS}}}=\sum_{k\in I}\left\langle
Se_{k},Te_{k}\right\rangle _{L^{2}\left(  \mathbb{R}^{d}\right)  }.
\]
Also, it is useful to observe that the inner product of arbitrary rank-one
operators in $L^{2}\left(  \mathbb{R}^{d}\right)  \otimes L^{2}\left(
\mathbb{R}^{d}\right)  $ is given by
\[
\left\langle u\otimes v,w\otimes y\right\rangle _{\mathcal{HS}}=\left\langle
u,w\right\rangle _{L^{2}\left(  \mathbb{R}^{d}\right)  }\left\langle
v,y\right\rangle _{L^{2}\left(  \mathbb{R}^{d}\right)  }.
\]

Let $\lambda$ be a linear functional in $\mathfrak{n}^{\ast}.$ Put%
\[
\lambda_{k}=\lambda\left(  Z_{k}\right)  .
\]
Treating the $\lambda_{k}$ as unknowns, we define
\begin{equation}
B\left(  \lambda\right)  =\left[
\begin{array}
[c]{ccc}%
\lambda\left[  X_{1},Y_{1}\right]  & \cdots & \lambda\left[  X_{1}%
,Y_{d}\right] \\
\vdots & \ddots & \vdots\\
\lambda\left[  X_{d},Y_{1}\right]  & \cdots & \lambda\left[  X_{d}%
,Y_{d}\right]
\end{array}
\right]  \label{B}%
\end{equation}
which is a square matrix of order $d$. The entries in $B(\lambda)$ are linear
combinations of the unknowns $\lambda_{1},\cdots,\lambda_{n-2d}.$ Computing
the determinant of the matrix $B\left(  \lambda\right)  ,$ we obtain a
polynomial in the unknowns $\lambda_{1},\cdots,\lambda_{n-2d}.$ Thus,
$\det\left[  \left[  X_{i},Y_{j}\right]  \right]  _{1\leq i,j\leq d}$ is a
non-vanishing homogeneous polynomial in the unknowns $Z_{1},\cdots,Z_{n-2d}$
if and only if $\det B\left(  \lambda\right)  $ is non-vanishing homogeneous
polynomial in the unknowns $\lambda_{1},\cdots,\lambda_{n-2d}.$ Therefore, if
the assumptions of Conditions \ref{conditiion} are met, then for $\lambda
\in\mathfrak{n}^{\ast},$ $\det\left(  B\left(  \lambda\right)  \right)  $ is a
non-vanishing polynomials in the unknowns $\lambda_{1},\cdots,\lambda_{n-2d}.$

\begin{proposition}
\label{analysis}Let $\mathfrak{n}$ be a Lie algebra over $%
\mathbb{R}
$ satisfying Condition \ref{conditiion} and let $L$ be the left regular
representation of the group $N.$

\begin{itemize}
\item The unitary dual of $N$ is parametrized by the smooth manifold
\[
\Sigma=\left\{
\begin{array}
[c]{c}%
\lambda\in\mathfrak{n}^{\ast}:\det\left(  B\left(  \lambda\right)  \right)
\neq0,\lambda\left(  X_{1}\right)  =\cdots\\
=\lambda\left(  X_{d}\right)  =\lambda\left(  Y_{1}\right)  =\cdots
=\lambda\left(  Y_{d}\right)  =0
\end{array}
\right\}
\]
which is naturally identified with a Zariski open subset of $\mathfrak{z}%
^{\ast}.$

\item Let $d\lambda$ be the Lebesgue measure on $\Sigma.$ The Plancherel
measure for the group $N$ is supported on $\Sigma$ and is equal to
\begin{equation}
d\mu\left(  \lambda\right)  =\left\vert \mathrm{\det}\left(  B\left(
\lambda\right)  \right)  \right\vert d\lambda.
\end{equation}

\item The unitary dual of $N$ which we denote by $\widehat{N}$ is up to a null
set equal to $\left\{  \pi_{\lambda}:\lambda\in\Sigma\right\}  $ where each
representation $\pi_{\lambda}$ is realized as acting in $L^{2}\left(
\mathbb{R}
^{d}\right)  $ such that
\begin{align*}
\pi_{\lambda}\left(  \exp\left(  \sum_{i=1}^{n-2d}z_{i}Z_{i}\right)  \right)
f\left(  t\right)   &  =e^{2\pi i\lambda\left(  \sum_{i=1}^{n-2d}z_{i}%
Z_{i}\right)  }f\left(  t\right)  ,\\
\pi_{\lambda}\left(  \exp\left(  \sum_{i=1}^{d}y_{i}Y_{i}\right)  \right)
f\left(  t\right)   &  =e^{-2\pi i\left\langle B\left(  \lambda\right)
y,t\right\rangle }f\left(  t\right)  ,\\
\pi_{\lambda}\left(  \exp\left(  \sum_{i=1}^{d}x_{i}X_{i}\right)  \right)
f\left(  t\right)   &  =f\left(  t-x\right)  ,
\end{align*}
where $y=\left(  y_{1},\cdots,y_{d}\right)  ^{tr},$ and $x=\left(
x_{1},\cdots,x_{d}\right)  .$

\item $L\cong\mathcal{P}\circ L\circ\mathcal{P}^{-1}=\int_{\Sigma}^{\oplus}%
\pi_{\lambda}\otimes\mathbf{1}_{L^{2}\left(  \mathbb{R}^{d}\right)  }%
d\mu\left(  \lambda\right)  $ and $\mathbf{1}_{L^{2}\left(  \mathbb{R}%
^{d}\right)  }$ is the identity operator on $L^{2}\left(  \mathbb{R}%
^{d}\right)  .$ Moreover for $\lambda\in\Sigma,$ we have
\[
\mathcal{P}(L(x)\phi)(\lambda)=\pi_{\lambda}(x)\circ\left(  \mathcal{P}%
\phi\right)  (\lambda).
\]

\end{itemize}
\end{proposition}

The results in the proposition above are some facts, which are well-known in
the theory of harmonic analysis of nilpotent Lie groups. See \cite{oussa},
where we specialized to the class of groups considered here. For general
nilpotent Lie groups, we refer the interested reader to Section $4.3$ in
\cite{Corwin} which contains a complete presentation of the Plancherel theory
of nilpotent Lie groups.

We will now provide a few examples of Lie groups satisfying Condition
\ref{conditiion}.

\begin{example}
\label{rot} Let $N$ be a nilpotent Lie group with Lie algebra $\mathfrak{n}$
spanned by the strong Malcev basis $Z_{1},$ $Z_{2},$ $Y_{1},$ $Y_{2},$
$X_{1},$ $X_{2}$ with non-trivial Lie brackets:%
\begin{align*}
\left[  X_{1},Y_{1}\right]   &  =Z_{1},\left[  X_{2},Y_{1}\right]  =-Z_{2},\\
\left[  X_{1},Y_{2}\right]   &  =Z_{2},\left[  X_{2},Y_{2}\right]  =Z_{1}.
\end{align*}
Clearly, $N$ satisfies all properties described in Condition \ref{conditiion}
and
\[
\det\left(  \left[  \left[  X_{i},Y_{j}\right]  \right]  _{1\leq i,j\leq
2}\right)  =\det\left[
\begin{array}
[c]{cc}%
Z_{1} & Z_{2}\\
-Z_{2} & Z_{1}%
\end{array}
\right]  =Z_{1}^{2}+Z_{2}^{2}.
\]
Applying Proposition \ref{faithful}, we define the monomorphism $\pi
:N\rightarrow GL_{7}\left(
\mathbb{R}
\right)  $ such that for
\[
p=\exp\left(  z_{1}Z_{1}\right)  \exp\left(  z_{2}Z_{2}\right)  \exp\left(
y_{1}Y_{1}\right)  \exp\left(  y_{2}Y_{2}\right)  \exp\left(  x_{1}%
X_{1}\right)  \exp\left(  x_{2}X_{2}\right)  ,
\]
the image of $p$ under the representation $\pi$ is the following matrix:%
\[
\left[
\begin{array}
[c]{ccccccc}%
1 & 0 & x_{1} & x_{2} & -y_{1} & -y_{2} & 2z_{1}\\
0 & 1 & -x_{2} & x_{1} & -y_{2} & y_{1} & 2z_{2}\\
0 & 0 & 1 & 0 & 0 & 0 & y_{1}\\
0 & 0 & 0 & 1 & 0 & 0 & y_{2}\\
0 & 0 & 0 & 0 & 1 & 0 & x_{1}\\
0 & 0 & 0 & 0 & 0 & 1 & x_{2}\\
0 & 0 & 0 & 0 & 0 & 0 & 1
\end{array}
\right]  .
\]
Next, referring to Proposition \ref{analysis}, the Plancherel measure is
supported on the manifold
\[
\Sigma=\left\{
\begin{array}
[c]{c}%
\lambda\in\mathfrak{n}^{\ast}:\lambda\left(  Z_{1}\right)  ^{2}+\lambda\left(
Z_{2}\right)  ^{2}\neq0,\\
\lambda\left(  Y_{j}\right)  =0,\lambda\left(  X_{j}\right)  =0\text{ for
}1\leq j\leq3
\end{array}
\right\}
\]
and the Plancherel measure is $\left\vert \lambda_{1}^{2}+\lambda_{2}%
^{2}\right\vert d\lambda_{1}d\lambda_{2}$ where $\lambda_{k}=\lambda\left(
Z_{k}\right)  .$
\end{example}

The following example exhausts all elements in the class of groups considered
in this paper.

\begin{example}
Fix two natural numbers $n$ and $d,$ such that $n-2d>0.$ Let $M$ be a matrix
of order $d$ with entries in $%
\mathbb{R}
Z_{1}\oplus\cdots\oplus%
\mathbb{R}
Z_{n-2d}$ such that $\det\left(  M\right)  $ is a non-vanishing homogeneous
polynomial in the unknowns $Z_{1},Z_{2},\cdots,Z_{n-2d}.$ Now let
$\mathfrak{a=%
\mathbb{R}
}$-span $\left\{  X_{1},\cdots,X_{d}\right\}  ,$ and $\mathfrak{b=%
\mathbb{R}
}$-span $\left\{  Y_{1},\cdots,Y_{d}\right\}  $ such that $\left[  X_{i}%
,Y_{j}\right]  =M_{i,j}$ and $M_{i,j}$ is the entry of $M$ located at the
intersection of the $i$-th row and $j$-th column. The Lie algebra
\[
\mathfrak{n=a\oplus b}\oplus\left(
\mathbb{R}
Z_{1}\oplus\cdots\oplus%
\mathbb{R}
Z_{n-2d}\right)
\]
satisfies all properties given in Condition \ref{conditiion}.
\end{example}

Now, we define $\Gamma_{\mathfrak{b}}=\exp\left(
\mathbb{Z}
Y_{1}+\cdots+%
\mathbb{Z}
Y_{d}\right)  ,$ $\Gamma_{\mathfrak{a}}=\exp\left(
\mathbb{Z}
X_{1}+\cdots+%
\mathbb{Z}
X_{d}\right)  $,
\[
\Gamma_{\mathfrak{z}}=\exp\left(
\mathbb{Z}
Z_{1}+\cdots+%
\mathbb{Z}
Z_{n-2d}\right)
\]
and
\begin{equation}
\Gamma=\Gamma_{\mathfrak{z}}\Gamma_{\mathfrak{b}}\Gamma_{\mathfrak{a}}\subset
N. \label{gammas}%
\end{equation}
Then $\Gamma$ is a discrete subset of $N\ $which is not generally a subgroup
of $N.$

\section{Overview of Main Results}

In this section, we will present an overview of the main results. In order to
do so, we will need a few important definitions.

\begin{definition}
We say a function $f\in L^{2}(N)$ is bandlimited if its Plancherel transform
is supported on a bounded measurable subset of $\Sigma.$ Fix a measurable
field of unit vectors $\mathbf{e=}\left\{  \mathbf{e}_{\lambda}\right\}
_{\lambda\in\Sigma}$ where $\mathbf{e}_{\lambda}\in L^{2}\left(
\mathbb{R}^{d}\right)  .$ We say a Hilbert space is a multiplicity-free
left-invariant subspace of $L^{2}\left(  N\right)  $ if
\[
\mathbf{H}\left(  \mathbf{e}\right)  \text{ }\mathbf{=}\text{ }\mathcal{P}%
^{-1}\left(  \int_{\Sigma}^{\oplus}L^{2}\left(  \mathbb{R}^{d}\right)
\otimes\mathbf{e}_{\lambda}\text{ }d\mu\left(  \lambda\right)  \right)  .
\]

\end{definition}

We observe here that the Hilbert space $\mathcal{P}\left(  \mathbf{H}\left(
\mathbf{e}\right)  \right)  $ is naturally identified with $L^{2}\left(
\Sigma\times\mathbb{R}^{d}\right)  .$ Next, we define
\begin{equation}
\mathbf{E}=\left\{  \lambda\in\mathfrak{z}^{\ast}:\left\vert \det B\left(
\lambda\right)  \right\vert \neq0,\text{ and }\left\vert \det B\left(
\lambda\right)  \right\vert \leq1\right\}  . \label{set E}%
\end{equation}
It is easy to see that $\mathbf{E}$ is the intersection of a Zariski open
subset of $\mathfrak{z}^{\ast}$ and a closed subset of $\mathfrak{z}^{\ast}.$
Also, $\mathbf{E}$ is not bounded in general and $\mathbf{E}$ is necessarily a
set of positive Lebesgue measure on $\mathfrak{z}^{\ast}.$ In order to develop
a theory of bandlimitation, we will need to consider some bounded subset of
$\mathbf{E.}$

For any given bounded set $\mathbf{A}\subset\Sigma,$ we define the
corresponding multiplicity-free, bandlimited, left-invariant Hilbert subspace
$\mathbf{H}\left(  \mathbf{e,A}\right)  $ as follows $\ $%
\begin{equation}
\mathbf{H}\left(  \mathbf{e,A}\right)  \text{ }\mathbf{=}\text{ }%
\mathcal{P}^{-1}\left(  \int_{\mathbf{A}}^{\oplus}L^{2}\left(
\mathbb{R}
^{d}\right)  \otimes\mathbf{e}_{\lambda}\left\vert \det B\left(
\lambda\right)  \right\vert d\lambda\right)  .\label{EC}%
\end{equation}
To be more precise, for any $\phi\in\mathbf{H}\left(  \mathbf{e,A}\right)  ,$
there exists a measurable field of vectors $\left\{  \mathbf{w}_{\lambda
}^{\phi}\right\}  _{\lambda\in\mathbf{A}},$ $\mathbf{w}_{\lambda}^{\phi}\in
L^{2}\left(
\mathbb{R}
^{d}\right)  $ such that
\[
\mathcal{P}\phi\left(  \lambda\right)  =\left\{
\begin{array}
[c]{c}%
\mathbf{w}_{\lambda}^{\phi}\otimes\mathbf{e}_{\lambda}\text{ if }\lambda
\in\mathbf{A}\\
0\text{ if }\lambda\notin\mathbf{A}%
\end{array}
\right.  .
\]
Let $\phi\in\mathbf{H}\left(  \mathbf{e,A}\right)  $ and define the linear map
$W_{\phi}:\mathbf{H}\left(  \mathbf{e,A}\right)  \mathbf{\rightarrow}$
$L^{2}\left(  N\right)  ,$ such that $W_{\phi}\psi\left(  x\right)
=\left\langle \psi,L\left(  x\right)  \phi\right\rangle .$ It is easy to see
that the space $W_{\phi}\left(  \mathbf{H}\left(  \mathbf{e,A}\right)
\right)  $ is a subspace of $L^{2}\left(  N\right)  $ which consists of
continuous functions.

Let $\iota:%
\mathbb{R}
^{n-2d}\rightarrow\mathfrak{z}^{\ast}$ be a map defined by
\[
\iota\left(  \lambda_{1},\cdots,\lambda_{n-2d}\right)  =%
{\displaystyle\sum\limits_{k=1}^{n-2d}}
\lambda_{k}Z_{k}^{\ast}%
\]
where $\left\{  Z_{k}^{\ast}:1\leq k\leq n-2d\right\}  $ is the dual basis of
$\mathfrak{z}^{\ast}$ which is associated to
\[
\left\{  Z_{k}:1\leq k\leq n-2d\right\}
\]
which is a fixed basis for the central ideal $\mathfrak{z.}$ Clearly, $\iota$
is a measurable bijection. Identifying $%
\mathbb{R}
^{n-2d}$ with $\mathfrak{z}^{\ast}$via the map $\iota,$ we slightly abuse
notation when we say $\mathfrak{z}^{\ast}=%
\mathbb{R}
^{n-2d}.$ In order to make a simpler presentation, we will adopt this abuse of
notation for the remainder of the paper. Now, let $\mathbf{C}\subset
\mathfrak{z}^{\ast}=%
\mathbb{R}
^{n-2d}$ be a bounded set such that
\[
\left\{  e^{2\pi i\left\langle k,\lambda\right\rangle }\chi_{\mathbf{C}%
}\left(  \lambda\right)  :k\in%
\mathbb{Z}
^{n-2d}\right\}
\]
is a Parseval frame for $L^{2}\left(  \mathbf{C,}d\lambda\right)  .$ For
example, it suffices to pick $\mathbf{C\subseteq I}$ such that the collection
$\left\{  \mathbf{I+}\text{ }k\text{ }\mathbf{:}\text{ }k\in%
\mathbb{Z}
^{n-2d}\right\}  $ forms a measurable partition of $%
\mathbb{R}
^{n-2d}=\mathfrak{z}^{\ast}$. Our main results are summarized as follows.

\begin{theorem}
\label{main}Let $N$ be a connected, simply connected nilpotent Lie group
satisfying Condition \ref{conditiion}. There exists $\phi\in\mathbf{H}\left(
\mathbf{e,E\cap C}\right)  $ such that $W_{\phi}\left(  \mathbf{H}\left(
\mathbf{e,E\cap C}\right)  \right)  $ is a $\Gamma$-sampling subspace of
$L^{2}\left(  N\right)  $ with sinc-type function $W_{\phi}(\phi).$ Moreover,
$W_{\phi}\left(  \mathbf{H}\left(  \mathbf{e,E\cap C}\right)  \right)  $ does
not generally have the interpolation property with respect to $\Gamma.$
\end{theorem}

Now, let $N$ be a nilpotent Lie group which satisfies all properties described
in Condition \ref{conditiion} such that additionally, there exist a strong
Malcev basis for the Lie algebra $\mathfrak{n},$ and a compact subset
$\mathbf{R\ }$of $\mathfrak{z}^{\ast}$ such that for $\mathbf{E}^{\circ
}=\mathbf{E\cap R,}$%

\[
\int_{\mathbf{E}^{\circ}}\left\vert \det B\left(  \lambda\right)  \right\vert
d\lambda=1.
\]
Since $\mathbf{E}^{\circ}$ is a subset of $\mathbf{E,}$ for each $\lambda
\in\mathbf{E}^{\circ}$ there exists\ (see Remark \ref{remark}) a corresponding
set $E\left(  \lambda\right)  $ which tiles $\mathbb{R}^{d}$ by $\mathbb{Z}%
^{d}$ and packs $\mathbb{R}^{d}$ by $B\left(  \lambda\right)  ^{-tr}%
\mathbb{Z}^{d}.$ Fix a fundamental domain $\Lambda$ for the lattice $%
\mathbb{Z}
^{n-2d}$ such that
\[
\mathbf{E}^{\circ}=%
{\displaystyle\bigcup\limits_{k_{j}\in S}}
\left(  \left(  \Lambda-\kappa_{j}\right)  \cap\mathbf{E}^{\circ}\right)
\]
where $S$ is a finite subset of $%
\mathbb{Z}
^{n-2d}$ and each $\left(  \Lambda-\kappa_{j}\right)  \cap\mathbf{E}^{\circ}$
is a set of positive Lebesgue measure on $\mathbb{R}^{n-2d}.$ For $\lambda\in\Sigma,$ we define the
map $\lambda\mapsto\mathbf{u}_{\lambda}$ on $\Sigma$ such that
\begin{equation}
\mathbf{u}_{\lambda}=\left\{
\begin{array}
[c]{c}%
\left\vert \det B\left(  \lambda\right)  \right\vert ^{1/2}\chi_{E\left(
\lambda\right)  }\text{ if }\lambda\in\mathbf{E}^{\circ}\\
0\text{ if }\lambda\notin\Sigma-\mathbf{E}^{\circ}%
\end{array}
\right.  \label{extension}%
\end{equation}
and $\phi\in L^{2}\left(  N\right)  $ such that
\begin{equation}
\mathcal{P}\phi\left(  \lambda\right)  =\frac{\mathbf{u}_{\lambda}%
\otimes\mathbf{e}_{\lambda}}{\sqrt{\left\vert \det B\left(  \lambda\right)
\right\vert }}.
\end{equation}

\begin{theorem}
\label{steps}If for each $1\leq j,j^{\prime}\leq\mathrm{card}\left(  S\right)
,$ $j\neq j^{\prime},$ for $\lambda\in\Lambda,$ for arbitrary functions
$\mathbf{f},\mathbf{g}\in L^{2}\left(  \mathbb{R}^{d}\right)  ,$ and for
distinct $\kappa_{j},\kappa_{j^{\prime}}\in S$
\[
\left(  \left\langle \mathbf{f},\pi_{\lambda-\kappa_{j}}\left(  \gamma
_{1}\right)  \mathbf{u}_{\lambda-\kappa_{j}}\right\rangle \right)
_{\gamma_{1}\in\Gamma_{\mathfrak{b}}\Gamma_{\mathfrak{a}}}\perp\left(
\left\langle \mathbf{g},\pi_{\lambda-\kappa_{j^{\prime}}}\left(  \gamma
_{1}\right)  \mathbf{u}_{\lambda-\kappa_{j^{\prime}}}\right\rangle \right)
_{\gamma_{1}\in\Gamma_{\mathfrak{b}}\Gamma_{\mathfrak{a}}}\text{,}%
\]
then $W_{\phi}\left(  \mathbf{H}\left(  \mathbf{e,E}^{\circ}\right)  \right)
$ is a $\Gamma$-sampling space with sinc-type function $W_{\phi}\left(
\phi\right)  .$ Moreover, $W_{\phi}\left(  \mathbf{H}\left(  \mathbf{e,E}%
^{\circ}\right)  \right)  $ has the interpolation property.
\end{theorem}

The proof of Theorem \ref{main} and Theorem \ref{steps} will be given in the
last section of this paper.

\section{Results on Frames and Orthonormal Bases}

We will need to be familiar with the theory of frames (see \cite{Pete},
\cite{Pfander} and \cite{Han Yang Wang}). Given a countable sequence $\left\{
f_{i}\right\}  _{i\in I}$ of vectors in a Hilbert space $\mathbf{H},$ we say
$\left\{  f_{i}\right\}  _{i\in I}$ forms a \textbf{frame} if and only if
there exist strictly positive real numbers $A,B$ such that for any vector
$f\in\mathbf{H}$,
\[
A\left\Vert f\right\Vert ^{2}\leq\sum_{i\in I}\left\vert \left\langle
f,f_{i}\right\rangle \right\vert ^{2}\leq B\left\Vert f\right\Vert ^{2}.
\]
In the case where $A=B$, the sequence of vectors $\left\{  f_{i}\right\}
_{i\in I}$ forms what is called a \textbf{tight frame}, and if $A=B=1$,
$\left\{  f_{i}\right\}  _{i\in I}$ is called a \textbf{Parseval frame}
because it satisfies the Parseval equality:%
\[
\sum_{i\in I}\left\vert \left\langle f,f_{i}\right\rangle \right\vert
^{2}=\left\Vert f\right\Vert ^{2}\text{ for all }f\in\mathbf{H.}%
\]
Also, if $\left\{  f_{i}\right\}  _{i\in I}$ is a Parseval frame such that for
all $i\in I,\left\Vert f_{i}\right\Vert =1$ then $\left\{  f_{i}\right\}
_{i\in I}$ is an orthonormal basis for $\mathbf{H}$.

Let $\Xi=A\mathbb{Z}^{2d}$ for some matrix $A$. We say $\Xi$ is a full-rank
lattice if $A$ is non-singular. We say a lattice is separable if
$\Xi=A\mathbb{Z}^{d}\times B\mathbb{Z}^{d}.\ $A \textbf{fundamental domain}
$D$ for a lattice in $\mathbb{R}^{d}$ is a measurable set which satisfies the
following: $(D+l)\cap(D+l^{\prime})\neq\emptyset$ for distinct $l,$
$l^{\prime}$ in $\Xi,$ and $\mathbb{R}^{d}={\bigcup\limits_{l\in\Xi}}\left(
D+l\right)  .$ We say $D$ is a \textbf{packing set} for $\Xi,$ if $\left(
D+l\right)  \cap\left(  D+l^{\prime}\right)  $ has Lebesgue measure zero for
any $l\neq l^{\prime}.$ Let $\Xi=A\mathbb{Z}^{d}\times B\mathbb{Z}^{d}$ be a
full-rank lattice in $\mathbb{R}^{2d}$ and $f\in L^{2}\left(  \mathbb{R}%
^{d}\right)  $. The family of functions in $L^{2}\left(  \mathbb{R}%
^{d}\right)  $,
\begin{equation}
\mathcal{G}\left(  f,A\mathbb{Z}^{d}\times B\mathbb{Z}^{d}\right)  =\left\{
e^{2\pi i\left\langle k,x\right\rangle }f\left(  x-n\right)  :k\in
B\mathbb{Z}^{d},n\in A\mathbb{Z}^{d}\right\}  \label{gab}%
\end{equation}
is called a \textbf{Gabor system}. Gabor frames are a particular type of frame
whose elements are generated by time-frequency shifts of a single vector. A
Gabor system which is a Parseval frame is called a \textbf{Gabor Parseval
frame}. Let $r$ be a natural number. Let $\Xi=A\mathbb{Z}^{r}$ be a full-rank
lattice in $\mathbb{R}^{r}$. The \textbf{volume} of $\Xi$ is defined as
\[
\mathrm{vol}\left(  \Xi\right)  =\left\vert \det A\right\vert ,
\]
and the \textbf{density} of the lattice $\Xi$ is defined as $d\left(
\Xi\right)  =\left\vert \det A\right\vert ^{-1}.$

\begin{lemma}
\label{dense}(\textbf{Density Condition}) Let $\Xi=A\mathbb{Z}^{d}\times
B\mathbb{Z}^{d}$ be a full-rank lattice in $\mathbb{R}^{2d}$. There exits
$f\in L^{2}(\mathbb{R}^{d})$ such that $\mathcal{G}\left(  f,A\mathbb{Z}%
^{d}\times B\mathbb{Z}^{d}\right)  $ is a Parseval frame in $L^{2}\left(
\mathbb{R}^{d}\right)  $ if and only if $\mathrm{vol}\left(  \Xi\right)
=\left\vert \det A\det B\right\vert \leq1.$
\end{lemma}

A proof of Lemma \ref{dense} is given in \cite{Han Yang Wang} Theorem $3.3.$

\begin{lemma}
\label{ONB copy(1)} Let $\Xi$ be a full-rank lattice in $\mathbb{R}^{2d}$.
There exists $f\in L^{2}\left(  \mathbb{R}^{d}\right)  $ such that
$\mathcal{G}\left(  f,\Xi\right)  $ is an orthonormal basis if and only if
$\mathrm{vol}\left(  \Xi\right)  =1.$ Also, if $\mathcal{G}\left(
f,\Xi\right)  $ is a Parseval frame for $L^{2}(\mathbb{R}^{d})$, then $\Vert
f\Vert^{2}=\mathrm{vol}(\Xi).$
\end{lemma}

Lemma \ref{ONB copy(1)} is due to Theorem $1.3$ and the proof of Lemma $3.2$
is given in \cite{Han Yang Wang}. Next, from the definition of the irreducible
representations of $N,$ provided in Proposition \ref{analysis}, it is easy to
see that for $f\in L^{2}\left(
\mathbb{R}
^{d}\right)  ,$ $\pi_{\lambda}\left(  \Gamma_{\mathfrak{b}}\Gamma
_{\mathfrak{a}}\right)  f$ is a Gabor system for each fixed $\lambda\in
\Sigma.$ Moreover, following the notation given in (\ref{gab}), we write:
\[
\pi_{\lambda}\left(  \Gamma_{\mathfrak{b}}\Gamma_{\mathfrak{a}}\right)
f\text{ }\mathbf{=}\text{ }\mathcal{G}\left(  f,\mathbb{Z}^{d}\times B\left(
\lambda\right)  \mathbb{Z}^{d}\right)  .
\]
We recall that the set $\mathbf{C}$ satisfies the following conditions:
$\mathbf{C}\subset\mathfrak{z}^{\ast}=%
\mathbb{R}
^{n-2d}$ is a bounded set such that the system
\[
\left\{  e^{2\pi i\left\langle k,\lambda\right\rangle }\chi_{\mathbf{C}%
}\left(  \lambda\right)  :k\in%
\mathbb{Z}
^{n-2d}\right\}
\]
is a Parseval frame for $L^{2}\left(  \mathbf{C,}d\lambda\right)  .$ Also, we
recall that
\[
\mathbf{E}=\left\{  \lambda\in\mathfrak{z}^{\ast}:\left\vert \det B\left(
\lambda\right)  \right\vert \neq0,\text{ and }\left\vert \det B\left(
\lambda\right)  \right\vert \leq1\right\}  .
\]
Therefore,
\[
\left\{  e^{2\pi i\left\langle k,\lambda\right\rangle }\chi_{\mathbf{E\cap C}%
}\left(  \lambda\right)  :k\in%
\mathbb{Z}
^{n-2d}\right\}
\]
is a Parseval frame for the Hilbert space $L^{2}\left(  \mathbf{E\cap
C},d\lambda\right)  .$

\begin{lemma}
\label{suff}There exists a function $\phi\in\mathbf{H}\left(  \mathbf{e,E\cap
C}\right)  $ such that $L\left(  \Gamma\right)  \phi$ is a Parseval frame in
$\mathbf{H}\left(  \mathbf{e,E\cap C}\right)  .$
\end{lemma}

\begin{proof}
We know that by the Density Condition (see Lemma \ref{dense}), for $\lambda
\in\mathbf{E\cap C},$ there exists rank-one operators $T_{\lambda}%
=\mathbf{u}_{\lambda}\otimes\mathbf{e}_{\lambda}$ such that
\begin{equation}
\mathcal{P}\phi\left(  \lambda\right)  =\frac{T_{\lambda}}{\sqrt{\left\vert
\det B\left(  \lambda\right)  \right\vert }},
\end{equation}
and the system $\mathcal{G}\left(  \mathbf{u}_{\lambda},%
\mathbb{Z}
^{d}\times B\left(  \lambda\right)
\mathbb{Z}
^{d}\right)  $ is a Gabor Parseval frame in $L^{2}\left(
\mathbb{R}
^{d}\right)  $. Next, given any vector $\psi\in\mathbf{H}\left(
\mathbf{e,E\cap C}\right)  ,$ we obtain that
\begin{align}
&  \sum_{\gamma\in\Gamma}\left\vert \left\langle \psi,L\left(  \gamma\right)
\phi\right\rangle _{\mathbf{H}\left(  \mathbf{e,E\cap C}\right)  }\right\vert
^{2}\nonumber\\
&  =\sum_{\gamma\in\Gamma}\left\vert \int_{\mathbf{E\cap C}}\left\langle
\mathcal{P}\psi\left(  \lambda\right)  ,\pi_{\lambda}\left(  \gamma\right)
\circ\mathcal{P}\phi\left(  \lambda\right)  \right\rangle _{\mathcal{HS}}%
d\mu\left(  \lambda\right)  \right\vert ^{2}\nonumber\\
&  =\sum_{\gamma\in\Gamma}\left\vert \int_{\mathbf{E\cap C}}\left\langle
\mathcal{P}\psi\left(  \lambda\right)  ,\pi_{\lambda}\left(  \gamma\right)
\left(  \left\vert \det B\left(  \lambda\right)  \right\vert ^{-1/2}%
\mathbf{u}_{\lambda}\right)  \otimes\mathbf{e}_{\lambda}\right\rangle
_{\mathcal{HS}}\left\vert \det B\left(  \lambda\right)  \right\vert
d\lambda\right\vert ^{2}\nonumber\\
&  =\sum_{\gamma\in\Gamma}\left\vert \int_{\mathbf{E\cap C}}\left\langle
\mathcal{P}\psi\left(  \lambda\right)  ,\pi_{\lambda}\left(  \gamma\right)
\mathbf{u}_{\lambda}\otimes\mathbf{e}_{\lambda}\right\rangle _{\mathcal{HS}%
}\left\vert \det B\left(  \lambda\right)  \right\vert ^{1/2}d\lambda
\right\vert ^{2}. \label{star}%
\end{align}
Using the fact that
\[
\left\{  e^{2\pi i\left\langle k,\lambda\right\rangle }\chi_{\mathbf{E\cap C}%
}\left(  \lambda\right)  :k\in%
\mathbb{Z}
^{n-2d}\right\}
\]
is a Parseval frame for $L^{2}\left(  \mathbf{E\cap C},d\lambda\right)  ,$ and
letting
\begin{equation}
\mathbf{f}\left(  \lambda\right)  =\left\langle \mathcal{P}\psi\left(
\lambda\right)  ,\pi_{\lambda}\left(  \gamma_{1}\right)  \mathbf{u}_{\lambda
}\otimes\mathbf{e}_{\lambda}\right\rangle _{\mathcal{HS}}\left\vert \det
B\left(  \lambda\right)  \right\vert ^{1/2}, \label{form}%
\end{equation}
we obtain%
\begin{align*}
\sum_{\gamma\in\Gamma}\left\vert \left\langle \psi,L\left(  \gamma\right)
\phi\right\rangle _{\mathbf{H}\left(  \mathbf{e,E\cap C}\right)  }\right\vert
^{2}  &  =\sum_{\gamma_{1}\in\Gamma_{\mathfrak{b}}\Gamma_{\mathfrak{a}}}%
{\displaystyle\sum\limits_{m\in\mathbb{Z}^{n-2d}}}
\left\vert \int_{\mathbf{E\cap C}}e^{-2\pi i\left\langle \lambda
,m\right\rangle }\mathbf{f}\left(  \lambda\right)  \right\vert ^{2}d\lambda\\
&  =\sum_{\gamma_{1}\in\Gamma_{\mathfrak{b}}\Gamma_{\mathfrak{a}}}%
{\displaystyle\sum\limits_{m\in\mathbb{Z}^{n-2d}}}
\left\vert \widehat{\mathbf{f}}\left(  m\right)  \right\vert ^{2}d\lambda\\
&  =\sum_{\gamma_{1}\in\Gamma_{\mathfrak{b}}\Gamma_{\mathfrak{a}}}\left\Vert
\widehat{\mathbf{f}}\right\Vert _{l^{2}\left(
\mathbb{Z}
^{n-2d}\right)  }^{2}\\
&  =\sum_{\gamma_{1}\in\Gamma_{\mathfrak{b}}\Gamma_{\mathfrak{a}}}\left\Vert
\mathbf{f}\right\Vert _{L^{2}\left(  \mathbf{E\cap C,}d\lambda\right)  }^{2}.
\end{align*}
The last equality above is due to the Plancherel Theorem on $l^{2}\left(
{\mathbb{Z}^{n-2d}}\right)  .$ Using Equation (\ref{form}), letting
$\mathcal{P}\psi\left(  \lambda\right)  =\mathbf{w}_{\lambda}^{\psi}%
\otimes\mathbf{e}_{\lambda},$ where $\mathbf{w}_{\lambda}^{\psi}%
=\mathbf{w}_{\lambda}\in L^{2}\left(
\mathbb{R}
^{d}\right)  ,$ and coming back to (\ref{star}), it follows that:
\begin{align*}
&  \sum_{\gamma\in\Gamma}\left\vert \left\langle \psi,L\left(  \gamma\right)
\phi\right\rangle _{\mathbf{H}\left(  \mathbf{e,E\cap C}\right)  }\right\vert
^{2}\\
&  =\sum_{\gamma\in\Gamma}\left\vert \int_{\mathbf{E\cap C}}\left\langle
\mathcal{P}\psi\left(  \lambda\right)  ,\pi_{\lambda}\left(  \gamma\right)
\mathbf{u}_{\lambda}\otimes\mathbf{e}_{\lambda}\right\rangle _{\mathcal{HS}%
}\left\vert \det B\left(  \lambda\right)  \right\vert ^{1/2}d\lambda
\right\vert ^{2}\\
&  =\sum_{\gamma_{1}\in\Gamma_{\mathfrak{b}}\Gamma_{\mathfrak{a}}}%
\int_{\mathbf{E\cap C}}\left\vert \left\langle \mathbf{w}_{\lambda}%
\otimes\mathbf{e}_{\lambda},\pi_{\lambda}\left(  \gamma_{1}\right)
\mathbf{u}_{\lambda}\otimes\mathbf{e}_{\lambda}\right\rangle _{\mathcal{HS}%
}\right\vert ^{2}\left\vert \det B\left(  \lambda\right)  \right\vert
d\lambda\\
&  =\int_{\mathbf{E\cap C}}\sum_{\gamma_{1}\in\Gamma_{\mathfrak{b}}%
\Gamma_{\mathfrak{a}}}\left\vert \left\langle \mathbf{w}_{\lambda}%
\otimes\mathbf{e}_{\lambda},\pi_{\lambda}\left(  \gamma_{1}\right)
\mathbf{u}_{\lambda}\otimes\mathbf{e}_{\lambda}\right\rangle _{\mathcal{HS}%
}\right\vert ^{2}\left\vert \det B\left(  \lambda\right)  \right\vert
d\lambda\\
&  =\int_{\mathbf{E\cap C}}\sum_{\gamma_{1}\in\Gamma_{\mathfrak{b}}%
\Gamma_{\mathfrak{a}}}\left\vert \left\langle \mathbf{w}_{\lambda}%
,\pi_{\lambda}\left(  \gamma_{1}\right)  \mathbf{u}_{\lambda}\right\rangle
_{L^{2}\left(  \mathbb{R}^{d}\right)  }\right\vert ^{2}\left\vert \det
B\left(  \lambda\right)  \right\vert d\lambda.
\end{align*}
Since $\pi_{\lambda}\left(  \Gamma_{\mathfrak{b}}\Gamma_{\mathfrak{a}}\right)
\mathbf{u}_{\lambda}$ is a Parseval Gabor frame for each fixed $\lambda
\in\mathbf{E\cap C}$%
\[
\sum_{\gamma_{1}\in\Gamma_{\mathfrak{b}}\Gamma_{\mathfrak{a}}}\left\vert
\left\langle \mathbf{w}_{\lambda},\pi_{\lambda}\left(  \gamma_{1}\right)
\mathbf{u}_{\lambda}\right\rangle _{L^{2}\left(  \mathbb{R}^{d}\right)
}\right\vert ^{2}=\left\Vert \mathbf{w}_{\lambda}\right\Vert _{L^{2}\left(
\mathbb{R}^{d}\right)  }^{2},
\]
then%
\begin{align*}
\sum_{\gamma\in\Gamma}\left\vert \left\langle \psi,L\left(  \gamma\right)
\phi\right\rangle _{\mathbf{H}\left(  \mathbf{e,E\cap C}\right)  }\right\vert
^{2}  &  =\int_{\mathbf{E\cap C}}\left\Vert \mathbf{w}_{\lambda}\right\Vert
_{L^{2}\left(  \mathbb{R}^{d}\right)  }^{2}\left\vert \det B\left(
\lambda\right)  \right\vert d\lambda\\
&  =\int_{\mathbf{E\cap C}}\left\Vert \mathcal{P}\psi\left(  \lambda\right)
\right\Vert _{\mathcal{HS}}^{2}\left\vert \det B\left(  \lambda\right)
\right\vert d\lambda\\
&  =\left\Vert \psi\right\Vert _{\mathbf{H}\left(  \mathbf{e,E\cap C}\right)
}^{2}.
\end{align*}
Finally, we obtain that $L\left(  \Gamma\right)  \phi$ is a Parseval frame in
$\mathbf{H}\left(  \mathbf{e,E\cap C}\right)  .$
\end{proof}

\begin{lemma}
\label{const} If $L\left(  \Gamma\right)  \phi$ is a Parseval frame in
$\mathbf{H}\left(  \mathbf{e,E\cap C}\right)  $ as described in Lemma
\ref{suff} and if
\[
\int_{\mathbf{E\cap C}}\left\vert \det B\left(  \lambda\right)  \right\vert
d\lambda=1
\]
then $L\left(  \Gamma\right)  \phi$ is an orthonormal basis.

\begin{proof}
Recall from Lemma \ref{suff} that for $\lambda\in\mathbf{E\cap C},$
\[
\mathcal{P}\phi\left(  \lambda\right)  =\left\vert \det B\left(
\lambda\right)  \right\vert ^{-1/2}\mathbf{u}_{\lambda}\otimes\mathbf{e}%
_{\lambda},
\]
such that $\mathcal{G}\left(  \mathbf{u}_{\lambda},%
\mathbb{Z}
^{d}\times B\left(  \lambda\right)
\mathbb{Z}
^{d}\right)  $ is a Gabor Parseval frame in $L^{2}\left(
\mathbb{R}
^{d}\right)  $. Referring to the proof of Theorem $1.3$ in \cite{Han Yang
Wang}, $\left\Vert \mathbf{u}_{\lambda}\right\Vert _{L^{2}\left(
\mathbb{R}
^{d}\right)  }^{2}=\left\vert \det B\left(  \lambda\right)  \right\vert $ for
$\lambda\in\mathbf{E\cap C}$. Now,
\begin{align*}
\left\Vert \phi\right\Vert _{\mathbf{H}\left(  \mathbf{e,E\cap C}\right)
}^{2}  &  =\int_{\mathbf{E\cap C}}\left\Vert \mathcal{P}\phi\left(
\lambda\right)  \right\Vert _{\mathcal{HS}}^{2}\left\vert \det B\left(
\lambda\right)  \right\vert d\lambda\\
&  =\int_{\mathbf{E\cap C}}\left\Vert \left\vert \det B\left(  \lambda\right)
\right\vert ^{-1/2}\mathbf{u}_{\lambda}\otimes\mathbf{e}_{\lambda}\right\Vert
_{\mathcal{HS}}^{2}\left\vert \det B\left(  \lambda\right)  \right\vert
d\lambda\\
&  =\int_{\mathbf{E\cap C}}\left\Vert \mathbf{u}_{\lambda}\otimes
\mathbf{e}_{\lambda}\right\Vert _{\mathcal{HS}}^{2}d\lambda\\
&  =\int_{\mathbf{E\cap C}}\left\Vert \mathbf{u}_{\lambda}\right\Vert
_{L^{2}\left(
\mathbb{R}
^{d}\right)  }^{2}d\lambda\\
&  =\int_{\mathbf{E\cap C}}\left\vert \det B\left(  \lambda\right)
\right\vert d\lambda\\
&  =\mu\left(  \mathbf{E\cap C}\right)  =1.
\end{align*}
Since any unit-norm Parseval frame is an orthonormal basis, the proof is completed.
\end{proof}
\end{lemma}

\begin{remark}
\label{remark}Theorem $3.3$ in \cite{Pfander} guarantees that for each
$\lambda\in\mathbf{E\cap C},$ it is possible to pick
\[
\mathbf{u}_{\lambda}=\left\vert \det B\left(  \lambda\right)  \right\vert
^{1/2}\chi_{E\left(  \lambda\right)  }%
\]
such that $E\left(  \lambda\right)  $ tiles $%
\mathbb{R}
^{d}$ by $%
\mathbb{Z}
^{d}$ and packs $%
\mathbb{R}
^{d}$ by $B\left(  \lambda\right)  ^{-tr}%
\mathbb{Z}
^{d}$ and $\mathcal{G}\left(  \mathbf{u}_{\lambda},%
\mathbb{Z}
^{d}\times B\left(  \lambda\right)
\mathbb{Z}
^{d}\right)  $ is a Gabor Parseval frame in $L^{2}\left(
\mathbb{R}
^{d}\right)  $.
\end{remark}

\section{Proof of Results and Examples}

Recall that
\[
\mathfrak{n=a\oplus b\oplus z},
\]
$\left[  \mathfrak{a},\mathfrak{b}\right]  \subseteq\mathfrak{z,}$
$\mathfrak{a},\mathfrak{b},\mathfrak{z}$ are abelian algebras, $\dim_{%
\mathbb{R}
}\left(  \mathfrak{a}\right)  =\dim_{%
\mathbb{R}
}\left(  \mathfrak{b}\right)  =d,$ and
\[
\det\left(  \left[  \left[  X_{i},Y_{j}\right]  \right]  _{1\leq i,j\leq
d}\right)
\]
is a non-vanishing polynomial in the unknowns $Z_{1},\cdots,Z_{n-2d}.$ Also,
we recall that
\[
B\left(  \lambda\right)  =\left[
\begin{array}
[c]{ccc}%
\lambda\left[  X_{1},Y_{1}\right]  & \cdots & \lambda\left[  X_{1}%
,Y_{d}\right] \\
\vdots & \ddots & \vdots\\
\lambda\left[  X_{d},Y_{1}\right]  & \cdots & \lambda\left[  X_{d}%
,Y_{d}\right]
\end{array}
\right]
\]
and%
\[
d\mu\left(  \lambda\right)  =\left\vert \det B\left(  \lambda\right)
\right\vert d\lambda.
\]
Moreover, the unitary dual of $N$ is parametrized by the smooth manifold
\[
\Sigma=\left\{
\begin{array}
[c]{c}%
\lambda\in\mathfrak{n}^{\ast}:\det\left(  B\left(  \lambda\right)  \right)
\neq0,\lambda\left(  X_{1}\right)  =\cdots\\
=\lambda\left(  X_{d}\right)  =\lambda\left(  Y_{1}\right)  =\cdots
=\lambda\left(  Y_{d}\right)  =0
\end{array}
\right\}
\]
which is naturally identified with a Zariski open subset of $\mathfrak{z}%
^{\ast}.$

\begin{definition}
Let $\left(  \pi,\mathbf{H}_{\pi}\right)  $ denote a strongly continuous
unitary representation of a locally compact group $G.$ We say that the
representation $\left(  \pi,\mathbf{H}_{\pi}\right)  $ is admissible if and
only if the map $W_{\phi}:\mathbf{H\rightarrow}$ $L^{2}\left(  G\right)  ,$
$W_{\phi}\psi\left(  x\right)  =\left\langle \psi,\pi\left(  x\right)
\phi\right\rangle $ defines an isometry of $\mathbf{H}$ into $L^{2}\left(
G\right)  ,$ and we say that $\phi$ is an admissible vector or a continuous wavelet.
\end{definition}

\begin{proposition}
\label{sinc1} Let $\Gamma$ be a discrete subset of $G.$ Let $\phi$ be an
admissible vector for $\left(  \pi,\mathbf{H}_{\pi}\right)  $ such that
$\pi\left(  \Gamma\right)  \phi$ is a Parseval frame for $\mathbf{H}_{\pi}.$
Then $\mathbf{K}=W_{\phi}\left(  \mathbf{H}_{\pi}\right)  $ is a $\Gamma
$-sampling space, and $W_{\phi}\left(  \phi\right)  $ is the associated
sinc-type function for $\mathbf{K.}$
\end{proposition}

See Proposition $2.54$ in \cite{Fuhr cont}.

\subsection{Proof of Theorem \ref{main}}

Since $\mathbf{C}\subset\mathfrak{z}^{\ast}=%
\mathbb{R}
^{n-2d}$ is a bounded set such that the system
\[
\left\{  e^{2\pi i\left\langle k,\lambda\right\rangle }\chi_{\mathbf{C}%
}\left(  \lambda\right)  :k\in%
\mathbb{Z}
^{n-2d}\right\}
\]
is a Parseval frame for $L^{2}\left(  \mathbf{C,}d\lambda\right)  $ and%
\[
\left\{  e^{2\pi i\left\langle k,\lambda\right\rangle }\chi_{\mathbf{E\cap C}%
}\left(  \lambda\right)  :k\in%
\mathbb{Z}
^{n-2d}\right\}
\]
is a Parseval frame for the Hilbert space $L^{2}\left(  \mathbf{E\cap
C},d\lambda\right)  ,$ then according to Lemma \ref{suff}, there exists a
function $\phi\in\mathbf{H}\left(  \mathbf{e,E\cap C}\right)$ such that
$L\left(  \Gamma\right)  \phi$ is a Parseval frame in $\mathbf{H}\left(
\mathbf{e,E\cap C}\right)  .$ In fact, for $\lambda\in\mathbf{E\cap C},$
$T_{\lambda}=\mathbf{u}_{\lambda}\otimes\mathbf{e}_{\lambda},$ we define
$\phi$ such that
\begin{equation}
\mathcal{P}\phi\left(  \lambda\right)  =\frac{T_{\lambda}}{\sqrt{\left\vert
\det B\left(  \lambda\right)  \right\vert }},
\end{equation}
and the system $\mathcal{G}\left(  \mathbf{u}_{\lambda},%
\mathbb{Z}
^{d}\times B\left(  \lambda\right)
\mathbb{Z}
^{d}\right)  $ is a Gabor Parseval frame in $L^{2}\left(
\mathbb{R}
^{d}\right)  $. Next,
\begin{align*}
\left\Vert \mathcal{P}\left(  \phi\right)  \left(  \lambda\right)  \right\Vert
_{\mathcal{HS}}^{2}  &  =\left\Vert \frac{\mathbf{u}_{\lambda}\otimes
\mathbf{e}_{\lambda}}{\sqrt{\left\vert \det B\left(  \lambda\right)
\right\vert }}\right\Vert _{\mathcal{HS}}^{2}\\
&  =\left\vert \det B\left(  \lambda\right)  \right\vert ^{-1}\left\Vert
\mathbf{u}_{\lambda}\right\Vert _{L^{2}\left(
\mathbb{R}
^{d}\right)  }^{2}\\
&  =1.
\end{align*}
Since $N$ is unimodular, and since $\mu\left(  \mathbf{E\cap C}\right)
<\infty$ then from \cite{Fuhr cont}, Page $127,$ $\left(  L,\mathbf{H}\left(
\mathbf{e,E\cap C}\right)  \right)  $ is an admissible representation of $N.$
Moreover, $\phi$ is an admissible vector for the representation $\left(
L,\mathbf{H}\left(  \mathbf{e,E\cap C}\right)  \right)  .$ Now, appealing to
Proposition \ref{sinc1}, then $\mathbf{K}=W_{\phi}\left(  \mathbf{H}\left(
\mathbf{e,E\cap C}\right)  \right)  $ is a sampling space, and $W_{\phi
}\left(  \phi\right)  $ is the associated sinc-type function for $\mathbf{K.}$
To see that in general, $W_{\phi}\left(  \mathbf{H}\left(  \mathbf{e,E\cap
C}\right)  \right)  $ does not have the interpolation property, it suffices to
observe that the condition
\[
\left\Vert \phi\right\Vert _{\mathbf{H}\left(  \mathbf{e,E\cap C}\right)
}^{2}=\mu\left(  \mathbf{E\cap C}\right)  =1
\]
does not always hold. This completes the proof of Theorem \ref{main}.

\subsection{Proof of Theorem \ref{steps}}

In order to prove Theorem \ref{steps}, we will need a series of lemmas first.
Let us assume throughout this subsection that there exist a basis for the Lie
algebra $\mathfrak{n}$ and a compact subset $\mathbf{R}$ of $\mathfrak{z}%
^{\ast}$ such that $\mu\left(  \mathbf{E}\cap\mathbf{R}\right)  =1.$ Put
$\mathbf{E}^{\circ}=\mathbf{E\cap R}$ and
\[
\mathbf{H}\left(  \mathbf{e,E}^{\circ}\right)  =\mathcal{P}^{-1}\left(
\int_{\mathbf{E}^{\circ}}^{\oplus}L^{2}\left(
\mathbb{R}
^{d}\right)  \otimes\mathbf{e}_{\lambda}\left\vert \det B\left(
\lambda\right)  \right\vert d\lambda\right)  .
\]

\begin{lemma}
Let $\mathbf{R}$ be given such that
\[
\int_{\mathbf{E}^{\circ}}\left\vert \det B\left(  \lambda\right)  \right\vert
d\lambda=1.
\]
Then the set $\mathbf{E}^{\circ}$ cannot be contained in a fundamental domain
of the lattice $%
\mathbb{Z}
^{n-2d}.$
\end{lemma}

\begin{proof}
Assume that%
\[
\mu\left(  \mathbf{E}^{\circ}\right)  =\int_{\mathbf{E}^{\circ}}\left\vert
\det B\left(  \lambda\right)  \right\vert d\lambda=1.
\]
We observe that the function $\lambda\mapsto\left\vert \det B\left(
\lambda\right)  \right\vert $ is a non-constant continuous function which is
bounded above by $1$ on $\mathbf{E}^{\circ}\mathbf{.}$ Therefore,
\[
1=\int_{\mathbf{E}^{\circ}}\left\vert \det B\left(  \lambda\right)
\right\vert d\lambda<\int_{\mathbf{E}^{\circ}}d\lambda=m\left(  \mathbf{E}%
^{\circ}\right)
\]
where $m$ is the Lebesgue measure on $\Sigma.$ By contradiction, let us assume
that $\mathbf{E}^{\circ}$ is contained in a fundamental domain of a lattice $%
\mathbb{Z}
^{n-2d}$. Then
\[
1<m\left(  \mathbf{E}^{\circ}\right)  \leq1
\]
and we reach a contradiction.
\end{proof}

\begin{lemma}
\label{j} There exists a finite partition of $\mathbf{E}^{\circ}$:
\[
\mathbf{P}=\left\{  A_{1},A_{2},\cdots,A_{\mathrm{card}\left(  \mathbf{P}%
\right)  }\right\}
\]
such that

\begin{enumerate}
\item $\mathbf{E}^{\circ}=%
{\displaystyle\bigcup\limits_{A_{j}\in\mathbf{P}}}
A_{j}$ and
\[
\mathbf{H}\left(  \mathbf{e,E}^{\circ}\right)  =%
{\displaystyle\bigoplus\limits_{j=1}^{\mathrm{card}\left(  \mathbf{P}\right)
}}
\mathbf{H}\left(  \mathbf{e,}A_{j}\right)  .
\]

\item For each $j$ where $1\leq j\leq\mathrm{card}\left(  \mathbf{P}\right)
,$ $A_{j}$ is contained in a fundamental domain for $%
\mathbb{Z}
^{n-2d}$.

\item For each $j$ where $1\leq j\leq\mathrm{card}\left(  \mathbf{P}\right)
,$ there exists a Parseval frame of the type $L\left(  \Gamma\right)  \phi
_{j}$ for the Hilbert space
\[
\mathbf{H}\left(  \mathbf{e,}A_{j}\right)  =\mathcal{P}^{-1}\left(
\int_{A_{j}}^{\oplus}L^{2}\left(
\mathbb{R}
^{d}\right)  \otimes\mathbf{e}_{\lambda}\left\vert \det B\left(
\lambda\right)  \right\vert d\lambda\right)  .
\]

\end{enumerate}
\end{lemma}

\begin{proof}
Parts $1$, $2$ are obviously true. For the proof for Part $3,$ we observe that
if $A_{j}$ is contained in a fundamental domain of $%
\mathbb{Z}
^{n-2d}$ then
\[
\left\{  e^{2\pi i\left\langle k,\lambda\right\rangle }\chi_{A_{j}}\left(
\lambda\right)  :k\in%
\mathbb{Z}
^{n-2d}\right\}
\]
is a Parseval frame for the Hilbert space $L^{2}\left(  A_{j}\mathbf{,}%
d\lambda\right)  .$ Thus Lemma \ref{suff} gives us Part $3$.
\end{proof}

\begin{lemma}
\label{one}For each $1\leq j\leq\mathrm{card}\left(  \mathbf{P}\right)  ,$ we
can construct a Parseval frame of the type $L\left(  \Gamma\right)  \phi_{j}$,
such that
\[
\left\Vert \sum_{j=1}^{\mathrm{card}\left(  \mathbf{P}\right)  }\phi
_{j}\right\Vert _{\mathbf{H}\left(  \mathbf{e,E}^{\circ}\right)  }^{2}=1.
\]

\begin{proof}
The construction of a Parseval frame for each $\mathbf{H}\left(
\mathbf{e,}A_{j}\right)  ,$ $1\leq j\leq\mathrm{card}\left(  \mathbf{P}%
\right)  $ of the type $L\left(  \Gamma\right)  \phi_{j}$ is given in Lemma
\ref{suff}, and%
\begin{align*}
\left\Vert \sum_{j=1}^{\mathrm{card}\left(  \mathbf{P}\right)  }\phi
_{j}\right\Vert _{\mathbf{H}\left(  \mathbf{e,E}^{\circ}\right)  }^{2}  &
=\sum_{j=1}^{\mathrm{card}\left(  \mathbf{P}\right)  }\left\Vert \phi
_{j}\right\Vert _{\mathbf{H}\left(  \mathbf{e,}A_{j}\right)  }^{2}\\
&  =\int_{\cup_{j=1}^{\mathrm{card}\left(  \mathbf{P}\right)  }A_{j}%
}\left\vert \det B\left(  \lambda\right)  \right\vert d\lambda\\
&  =\int_{\mathbf{E}^{\circ}}\left\vert \det B\left(  \lambda\right)
\right\vert d\lambda\\
&  =1.
\end{align*}

\end{proof}
\end{lemma}

\begin{lemma}
Let $\phi=\sum_{j=1}^{\mathrm{card}\left(  \mathbf{P}\right)  }\phi_{j}$ such
that for each $1\leq j\leq\mathrm{card}\left(  \mathbf{P}\right)  ,$ $L\left(
\Gamma\right)  \phi_{j}$ is a Parseval frame for $\mathbf{H}\left(
\mathbf{e,}A_{j}\right)  $ and $\left\Vert \phi\right\Vert _{\mathbf{H}\left(
\mathbf{e,E}^{\circ}\right)  }^{2}=1.$ If $L\left(  \Gamma\right)  \left(
\phi\right)  $ is a Parseval frame then $L\left(  \Gamma\right)  \phi$ is an
orthonormal basis for $\mathbf{H}\left(  \mathbf{e,E}^{\circ}\right)  .$
\end{lemma}

\begin{proof}
If $L\left(  \Gamma\right)  \phi$ is a Parseval frame for $\mathbf{H}\left(
\mathbf{e,E}^{\circ}\right)  $ then $L\left(  \Gamma\right)  \phi$ must be an
orthonormal basis since $\left\Vert \phi\right\Vert _{\mathbf{H}\left(
\mathbf{e,E\cap C}\right)  }^{2}=1.$
\end{proof}

We would like to remark that in general the direct sum of Parseval frames is
not a Parseval frame. Next, let us fix a fundamental domain $\Lambda$ of $%
\mathbb{Z}
^{n-2d}$ such that
\[
\mathbf{E}^{\circ}=%
{\displaystyle\bigcup\limits_{\kappa_{j}\in S}}
\left(  \left(  \Lambda-\kappa_{j}\right)  \cap\mathbf{E}^{\circ}\right)  ,
\]
each $A_{j}=\mathbf{E}^{\circ}\cap\left(  \Lambda-\kappa_{j}\right)  $ is a
set of positive Lebesgue measure for all $\kappa_{j}\in S$ and $S$ is a finite
subset of $%
\mathbb{Z}
^{n-2d}.$ Clearly, the collection of sets
\[
\mathbf{P}=\left\{  A_{j}:1\leq j\leq\mathrm{card}\left(  S\right)  \right\}
\]
provides us with a partition of $\mathbf{E}^{\circ}$ as described in Lemma
\ref{j}.

\begin{lemma}
\label{phi}For each $1\leq j\leq\mathrm{card}\left(  \mathbf{P}\right)  ,$
there exists $\phi_{j}\in\mathbf{H}\left(  \mathbf{e,}A_{j}\right)  $ such
that the following holds.

\begin{enumerate}
\item $L\left(  \Gamma\right)  \phi_{j}$ is a Parseval frame for
$\mathbf{H}\left(  \mathbf{e,}A_{j}\right)  ,$

\item $\mathcal{P}\phi_{j}\left(  \lambda\right)  =\left(  \mathbf{u}%
_{\lambda}\otimes\mathbf{e}_{\lambda}\right)  \left\vert \det B\left(
\lambda\right)  \right\vert ^{-1/2}$ where 
\begin{equation}
\mathbf{u}_{\lambda}=\left\{
\begin{array}
[c]{c}%
\left\vert \det B\left(  \lambda\right)  \right\vert ^{1/2}\chi_{E\left(
\lambda\right)  }\text{ if }\lambda\in A_j\\
0\text{ if }\lambda\notin\Sigma-A_j%
\end{array}
\right.
\end{equation}

such that $E\left(  \lambda\right)  $ tiles $%
\mathbb{R}
^{d}$ by $%
\mathbb{Z}
^{d}$ and packs $%
\mathbb{R}
^{d}$ by $B\left(  \lambda\right)  ^{-tr}%
\mathbb{Z}
^{d}.$
\end{enumerate}
\end{lemma}

\begin{proof}
See Lemma \ref{suff} and Remark \ref{remark}.
\end{proof}

Let us now define
\[
\phi=\phi_{1}+\cdots+\phi_{\mathrm{card}\left(  \mathbf{P}\right)  }%
\]
such that each $\phi_{j}$ is as described in Lemma \ref{phi}. Then clearly,
\begin{equation}
\mathcal{P}\phi\left(  \lambda\right)  =\left\{
\begin{array}
[c]{c}%
\chi_{E\left(  \lambda\right)  }\otimes\mathbf{e}_{\lambda}\text{ if }%
\lambda\in\mathbf{E}^{\circ}\\
0\text{ if }\lambda\notin\Sigma-\mathbf{E}^{\circ}%
\end{array}
\right.  .
\end{equation}

\begin{lemma}
\label{ONB} If for each $1\leq j,j^{\prime}\leq\mathrm{card}\left(
\mathbf{P}\right)  ,$ $j\neq j^{\prime},$ and for arbitrary functions
$\mathbf{f},\mathbf{g}\in L^{2}\left(
\mathbb{R}
^{d}\right)  ,\ \kappa_{j},\kappa_{j^{\prime}}\in S$
\[
\left(  \left\langle \mathbf{f},\pi_{\lambda-\kappa_{j}}\left(  \gamma
_{1}\right)  \mathbf{u}_{\lambda-\kappa_{j}}\right\rangle \right)
_{\gamma_{1}\in\Gamma_{\mathfrak{b}}\Gamma_{\mathfrak{a}}}\perp\left(
\left\langle \mathbf{g},\pi_{\lambda-\kappa_{j^{\prime}}}\left(  \gamma
_{1}\right)  \mathbf{u}_{\lambda-\kappa_{j^{\prime}}}\right\rangle \right)
_{\gamma_{1}\in\Gamma_{\mathfrak{b}}\Gamma_{\mathfrak{a}}}\text{ }\
\]
for $\lambda\in\Lambda$ then $L\left(  \Gamma\right)  \left(  \phi\right)  $
is an orthonormal basis for the Hilbert space $\mathbf{H}\left(
\mathbf{e,E}^{\circ}\right)  .$
\end{lemma}

\begin{proof}
Let $\psi$ be any arbitrary element in
\[
\mathbf{H}\left(  \mathbf{e,E}^{\circ}\right)  =%
{\displaystyle\bigoplus\limits_{j=1}^{\mathrm{card}\left(  \mathbf{P}\right)
}}
\mathbf{H}\left(  \mathbf{e,}A_{j}\right)
\]
such that $\psi=\sum_{j=1}^{\mathrm{card}\left(  \mathbf{P}\right)  }\psi
_{j},$ for $\psi_{j}\in\mathbf{H}\left(  \mathbf{e,}A_{j}\right)  .$ Let
$r\left(  \lambda\right)  =\left\vert \det B\left(  \lambda\right)
\right\vert .$ Then
\[
\left\Vert \psi_{j}\right\Vert _{\mathbf{H}\left(  \mathbf{e,}A_{j}\right)
}^{2}=\int_{A_{j}}\left\Vert \mathcal{P}\psi_{j}\left(  \sigma\right)
\right\Vert _{\mathcal{HS}}^{2}r\left(  \sigma\right)  d\sigma\text{.}%
\]
Next, it is easy to see that%
\[
\left\Vert \psi_{j}\right\Vert _{\mathbf{H}\left(  \mathbf{e,}A_{j}\right)
}^{2}=\int_{\Lambda\text{ }}\left\Vert \mathcal{P}\psi_{j}\left(
\lambda-\kappa_{j}\right)  \right\Vert _{\mathcal{HS}}^{2}r\left(
\lambda-\kappa_{j}\right)  d\lambda.
\]
Let $\mathcal{P}\psi_{j}\left(  \lambda-\kappa_{j}\right)  =\left(
\mathbf{w}_{\lambda-\kappa_{j}}\otimes\mathbf{e}_{\lambda-\kappa_{j}}\right)
\in L^{2}\left(
\mathbb{R}
^{d}\right)  \otimes\mathbf{e}_{\lambda-\kappa_{j}}$ for $\lambda\in\Lambda.$
Then%
\begin{align}
&  \sum_{j=1}^{\mathrm{card}\left(  \mathbf{P}\right)  }\left\Vert \psi
_{j}\right\Vert _{\mathbf{H}\left(  \mathbf{e,}A_{j}\right)  }^{2}\label{c1}\\
&  =\sum_{j=1}^{\mathrm{card}\left(  \mathbf{P}\right)  }\int_{\Lambda\text{
}}\sum_{\gamma_{1}\in\Gamma_{\mathfrak{b}}\Gamma_{\mathfrak{a}}}\left\vert
\left\langle \mathbf{w}_{\lambda-\kappa_{j}},\pi_{\lambda-\kappa_{j}}\left(
\gamma_{1}\right)  \mathbf{u}_{\lambda-\kappa_{j}}\right\rangle \right\vert
^{2}r\left(  \lambda-\kappa_{j}\right)  d\lambda\\
&  =\int_{\Lambda\text{ }}\sum_{\gamma_{1}\in\Gamma_{\mathfrak{b}}%
\Gamma_{\mathfrak{a}}}\sum_{j=1}^{\mathrm{card}\left(  \mathbf{P}\right)
}\left\vert \left\langle \mathbf{w}_{\lambda-\kappa_{j}},\pi_{\lambda
-\kappa_{j}}\left(  \gamma_{1}\right)  \mathbf{u}_{\lambda-\kappa_{j}%
}\right\rangle \right\vert ^{2}r\left(  \lambda-\kappa_{j}\right)  d\lambda.
\label{c2}%
\end{align}
We would like to be able to state that for $\lambda\in\mathbf{E}^{\circ},$
\begin{align}
&  \sum_{\gamma_{1}\in\Gamma_{\mathfrak{b}}\Gamma_{\mathfrak{a}}}\sum
_{j=1}^{\mathrm{card}\left(  \mathbf{P}\right)  }\left\vert \left\langle
\mathbf{w}_{\lambda-\kappa_{j}},\pi_{\lambda-\kappa_{j}}\left(  \gamma
_{1}\right)  \mathbf{u}_{\lambda-\kappa_{j}}\right\rangle \right\vert
^{2}\label{summations}\\
&  =\sum_{\gamma_{1}\in\Gamma_{\mathfrak{b}}\Gamma_{\mathfrak{a}}}\left\vert
\sum_{j=1}^{\mathrm{card}\left(  \mathbf{P}\right)  }\left\langle
\mathbf{w}_{\lambda-\kappa_{j}},\pi_{\lambda-\kappa_{j}}\left(  \gamma
_{1}\right)  \mathbf{u}_{\lambda-\kappa_{j}}\right\rangle \right\vert ^{2}.
\label{summations2}%
\end{align}
Indeed, letting $\left(  b_{\gamma_{1}}\left(  \lambda\right)  \right)
_{\gamma_{1}\in\Gamma_{\mathfrak{b}}\Gamma_{\mathfrak{a}}}\in l^{2}\left(
\Gamma_{\mathfrak{b}}\Gamma_{\mathfrak{a}}\right)  $ such that $\left(
b_{\gamma_{1}}\left(  \lambda\right)  \right)  _{\gamma_{1}\in\Gamma
_{\mathfrak{b}}\Gamma_{\mathfrak{a}}}$ is a sum of $\mathrm{card}\left(
\mathbf{P}\right)  $-many sequences of the type $\left(  b_{\gamma_{1}}%
^{j}\left(  \lambda\right)  \right)  _{\gamma_{1}\in\Gamma_{\mathfrak{b}%
}\Gamma_{\mathfrak{a}}}$ such that%
\begin{align*}
\left(  b_{\gamma_{1}}\left(  \lambda\right)  \right)  _{\gamma_{1}\in
\Gamma_{\mathfrak{b}}\Gamma_{\mathfrak{a}}}  &  =\sum_{j=1}^{\mathrm{card}%
\left(  \mathbf{P}\right)  }\left(  \left\langle \mathbf{w}_{\lambda
-\kappa_{j}},\pi_{\lambda-\kappa_{j}}\left(  \gamma_{1}\right)  \mathbf{u}%
_{\lambda-\kappa_{j}}\right\rangle \right)  _{\gamma_{1}\in\Gamma
_{\mathfrak{b}}\Gamma_{\mathfrak{a}}}\\
&  =\sum_{j=1}^{\mathrm{card}\left(  \mathbf{P}\right)  }\left(  b_{\gamma
_{1}}^{j}\left(  \lambda\right)  \right)  _{\gamma_{1}\in\Gamma_{\mathfrak{b}%
}\Gamma_{\mathfrak{a}}},
\end{align*}
we compute the norm of the sequence $\left(  b_{\gamma_{1}}\left(
\lambda\right)  \right)  _{\gamma_{1}\in\Gamma_{\mathfrak{b}}\Gamma
_{\mathfrak{a}}}$ in two different ways. First,
\begin{align*}
\left\Vert \left(  b_{\gamma_{1}}\left(  \lambda\right)  \right)  _{\gamma
_{1}\in\Gamma_{\mathfrak{b}}\Gamma_{\mathfrak{a}}}\right\Vert ^{2}  &
=\sum_{\gamma_{1}\in\Gamma_{\mathfrak{b}}\Gamma_{\mathfrak{a}}}\left\vert
b_{\gamma_{1}}\left(  \lambda\right)  \right\vert ^{2}\\
&  =\sum_{\gamma_{1}\in\Gamma_{\mathfrak{b}}\Gamma_{\mathfrak{a}}}\left\vert
\sum_{j=1}^{\mathrm{card}\left(  \mathbf{P}\right)  }\left\langle
\mathbf{w}_{\lambda-\kappa_{j}},\pi_{\lambda-\kappa_{j}}\left(  \gamma
_{1}\right)  \mathbf{u}_{\lambda-\kappa_{j}}\right\rangle \right\vert ^{2}.
\end{align*}
Second
\begin{align*}
&  \left\Vert \left(  b_{\gamma_{1}}\left(  \lambda\right)  \right)
_{\gamma_{1}\in\Gamma_{\mathfrak{b}}\Gamma_{\mathfrak{a}}}\right\Vert ^{2}\\
&  =\left\Vert \sum_{j=1}^{\mathrm{card}\left(  \mathbf{P}\right)  }\left(
b_{\gamma_{1}}^{j}\left(  \lambda\right)  \right)  _{\gamma_{1}\in
\Gamma_{\mathfrak{b}}\Gamma_{\mathfrak{a}}}\right\Vert ^{2}\\
&  =\sum_{j=1}^{\mathrm{card}\left(  \mathbf{P}\right)  }\left\Vert \left(
b_{\gamma_{1}}^{j}\left(  \lambda\right)  \right)  _{\gamma_{1}\in
\Gamma_{\mathfrak{b}}\Gamma_{\mathfrak{a}}}\right\Vert ^{2}\\
&  =\sum_{j=1}^{\mathrm{card}\left(  \mathbf{P}\right)  }\left\Vert \left(
\left\langle \mathbf{w}_{\lambda-\kappa_{j}},\pi_{\lambda-\kappa_{j}}\left(
\gamma_{1}\right)  \mathbf{u}_{\lambda-\kappa_{j}}\right\rangle \right)
_{\gamma_{1}\in\Gamma_{\mathfrak{b}}\Gamma_{\mathfrak{a}}}\right\Vert ^{2}\\
&  =\sum_{j=1}^{\mathrm{card}\left(  \mathbf{P}\right)  }\sum_{\gamma_{1}%
\in\Gamma_{\mathfrak{b}}\Gamma_{\mathfrak{a}}}\left\vert \left\langle
\mathbf{w}_{\lambda-\kappa_{j}},\pi_{\lambda-\kappa_{j}}\left(  \gamma
_{1}\right)  \mathbf{u}_{\lambda-\kappa_{j}}\right\rangle \right\vert ^{2}.
\end{align*}
The second equality above is due to the fact that we assume that for $j\neq
j^{\prime},$%
\begin{equation}
\left(  \left\langle \mathbf{w}_{\lambda-\kappa_{j}},\pi_{\lambda-\kappa_{j}%
}\left(  \gamma_{1}\right)  \mathbf{u}_{\lambda-\kappa_{j}}\right\rangle
\right)  _{\gamma_{1}\in\Gamma_{\mathfrak{b}}\Gamma_{\mathfrak{a}}}%
\perp\left(  \left\langle \mathbf{w}_{\lambda-\kappa_{j^{\prime}}}%
,\pi_{\lambda-\kappa_{j}}\left(  \gamma_{1}\right)  \mathbf{u}_{\lambda
-\kappa_{j^{\prime}}}\right\rangle \right)  _{\gamma_{1}\in\Gamma
_{\mathfrak{b}}\Gamma_{\mathfrak{a}}}\text{ for }\lambda\in\Lambda.
\end{equation}
Thus, the equality given in (\ref{summations}) and (\ref{summations2}) holds.
Next, coming back to (\ref{c2}), we obtain
\[
\sum_{j=1}^{\mathrm{card}\left(  \mathbf{P}\right)  }\left\Vert \psi
_{j}\right\Vert _{\mathbf{H}\left(  \mathbf{e,}A_{j}\right)  }^{2}%
=\sum_{\gamma_{1}\in\Gamma_{\mathfrak{b}}\Gamma_{\mathfrak{a}}}\int
_{\Lambda\text{ }}\left\vert a_{\gamma_{1}}\left(  \lambda\right)  \right\vert
^{2}d\lambda
\]
where
\[
a_{\gamma_{1}}\left(  \lambda\right)  =\sum_{j=1}^{\mathrm{card}\left(
\mathbf{P}\right)  }\left\langle \mathbf{w}_{\lambda-\kappa_{j}}%
,\overset{\mathbf{q}_{j,\gamma_{1}}\left(  \lambda\right)  }{\overbrace
{\pi_{\lambda-\kappa_{j}}\left(  \gamma_{1}\right)  \mathbf{u}_{\lambda
-\kappa_{j}}r\left(  \lambda-\kappa_{j}\right)  ^{1/2}}}\right\rangle .
\]
Writing $\lambda\left(  m\right)  =e^{2\pi i\left\langle \lambda
,m\right\rangle },$ it follows that%
\begin{align*}
\sum_{j=1}^{\mathrm{card}\left(  \mathbf{P}\right)  }\left\Vert \psi
_{j}\right\Vert _{\mathbf{H}\left(  \mathbf{e,}A_{j}\right)  }^{2}  &
=\sum_{\gamma_{1}\in\Gamma_{\mathfrak{b}}\Gamma_{\mathfrak{a}}}\left\Vert
a_{\gamma_{1}}\right\Vert _{L^{2}\left(  \Lambda\right)  }^{2}\\
&  =\sum_{\gamma_{1}\in\Gamma_{\mathfrak{b}}\Gamma_{\mathfrak{a}}}\left\Vert
\widehat{a_{\gamma_{1}}}\right\Vert _{l^{2}\left(
\mathbb{Z}
^{n-2d}\right)  }^{2}\\
&  =\sum_{\gamma_{1}\in\Gamma_{\mathfrak{b}}\Gamma_{\mathfrak{a}}}\sum_{m\in%
\mathbb{Z}
^{n-2d}}\left\vert \widehat{a_{\gamma_{1}}}\left(  m\right)  \right\vert
^{2}\\
&  =\sum_{\gamma_{1}\in\Gamma_{\mathfrak{b}}\Gamma_{\mathfrak{a}}}\sum_{m\in%
\mathbb{Z}
^{n-2d}}\left\vert \int_{\Lambda\text{ }}a_{\gamma_{1}}\left(  \lambda\right)
\lambda\left(  m\right)  d\lambda\right\vert ^{2}\\
&  =\sum_{\gamma_{1}\in\Gamma_{\mathfrak{b}}\Gamma_{\mathfrak{a}}}\sum_{m\in%
\mathbb{Z}
^{n-2d}}\left\vert \sum_{j=1}^{\mathrm{card}\left(  \mathbf{P}\right)  }%
\int_{\Lambda\text{ }}\left\langle \mathbf{w}_{\lambda-\kappa_{j}}%
,\mathbf{q}_{j,\gamma_{1}}\left(  \lambda\right)  \right\rangle \lambda\left(
m\right)  d\lambda\right\vert ^{2}.
\end{align*}
Next, letting
\[
\overline{\lambda\left(  m\right)  }=\overline{e^{2\pi i\left\langle
\lambda,m\right\rangle }}=e^{-2\pi i\left\langle \lambda,m\right\rangle },
\]
we obtain that
\begin{align*}
\sum_{j=1}^{\mathrm{card}\left(  \mathbf{P}\right)  }\left\Vert \psi
_{j}\right\Vert _{\mathbf{H}\left(  \mathbf{e,}A_{j}\right)  }^{2}  &
=\sum_{\gamma_{1}\in\Gamma_{\mathfrak{b}}\Gamma_{\mathfrak{a}}}\sum_{m\in%
\mathbb{Z}
^{n-2d}}\left\vert \int_{\Lambda\text{ }}\sum_{j=1}^{\mathrm{card}\left(
\mathbf{P}\right)  }\left\langle \mathbf{w}_{\lambda-\kappa_{j}}%
,\overline{\lambda\left(  m\right)  }\mathbf{q}_{j,\gamma_{1}}\left(
\lambda\right)  \right\rangle d\lambda\right\vert ^{2}\\
&  =%
{\displaystyle\sum\limits_{\gamma\in\Gamma}}
\left\vert \int_{\mathbf{E}^{\circ}}\left\langle \mathcal{P}\psi\left(
\sigma\right)  ,\pi_{\sigma}\left(  \gamma\right)  \overset{\mathcal{P}%
\phi\left(  \sigma\right)  }{\overbrace{\left\vert \sigma\right\vert
^{-1/2}\mathbf{u}_{\sigma}\otimes\mathbf{e}_{\sigma}}}\right\rangle
_{\mathcal{HS}}\left\vert \sigma\right\vert d\sigma\right\vert ^{2}\\
&  =%
{\displaystyle\sum\limits_{\gamma\in\Gamma}}
\left\vert \int_{\mathbf{E}^{\circ}}\left\langle \mathcal{P}\psi\left(
\sigma\right)  ,\pi_{\sigma}\left(  \gamma\right)  \mathcal{P}\phi\left(
\sigma\right)  \right\rangle _{\mathcal{HS}}\left\vert \sigma\right\vert
d\sigma\right\vert ^{2}\\
&  =%
{\displaystyle\sum\limits_{\gamma\in\Gamma}}
\left\vert \left\langle \psi,L\left(  \gamma\right)  \phi\right\rangle
_{\mathbf{H}\left(  \mathbf{e,E}^{\circ}\right)  }\right\vert ^{2}.
\end{align*}
Finally, we arrive at this fact:
\[
\left\Vert \psi\right\Vert _{\mathbf{H}\left(  \mathbf{e,E}^{\circ}\right)
}^{2}=%
{\displaystyle\sum\limits_{\gamma\in\Gamma}}
\left\vert \left\langle \psi,L\left(  \gamma\right)  \phi\right\rangle
_{\mathbf{H}\left(  \mathbf{e,E}^{\circ}\right)  }\right\vert ^{2}.
\]
Thus, $L\left(  \Gamma\right)  \phi$ is a Parseval frame for $\mathbf{H}%
\left(  \mathbf{e,E}^{\circ}\right)  $. Now, we compute the norm of the vector
$\phi.$ Since
\[
\mathcal{P}\phi\left(  \lambda\right)  =\left(  \mathbf{u}_{\lambda}%
\otimes\mathbf{e}_{\lambda}\right)  \left\vert \det B\left(  \lambda\right)
\right\vert ^{-1/2},\text{ and }\mathbf{u}_{\lambda}=\left\vert \det B\left(
\lambda\right)  \right\vert ^{1/2}\chi_{E\left(  \lambda\right)  }%
\]
then
\[
\mathcal{P}\phi\left(  \lambda\right)  =\left(  \left\vert \det B\left(
\lambda\right)  \right\vert ^{1/2}\chi_{E\left(  \lambda\right)  }%
\otimes\mathbf{e}_{\lambda}\right)  \left\vert \det B\left(  \lambda\right)
\right\vert ^{-1/2}=\chi_{E\left(  \lambda\right)  }\otimes\mathbf{e}%
_{\lambda}.
\]
Since $E\left(  \lambda\right)  $ is a fundamental domain for $%
\mathbb{Z}
^{d},$ it follows that
\[
\left\Vert \phi\right\Vert _{\mathbf{H}\left(  \mathbf{e,E}^{\circ}\right)
}^{2}=1.
\]
Finally, because $L$ is a unitary representation, and using the fact that
$L\left(  \Gamma\right)  \phi$ forms a unit norm Parseval frame, then
$L\left(  \Gamma\right)  \phi$ forms an orthonormal basis in $\mathbf{H}%
\left(  \mathbf{e,E}^{\circ}\right)  .$
\end{proof}

\begin{lemma}
\label{admissible}If $\phi$ satisfies all conditions given in Lemma \ref{ONB},
then $\left\Vert \mathcal{P}\phi\left(  \lambda\right)  \right\Vert
_{\mathcal{HS}}=1$ for $\lambda\in\mathbf{E}^{\circ}$ and $\phi$ is an
admissible vector for the representation $\left(  L,\mathbf{H}\left(
\mathbf{e,E}^{\circ}\right)  \right)  .$
\end{lemma}

\begin{proof}
For any given $\lambda\in\mathbf{E}^{\circ},$
\begin{align*}
\left\Vert \mathcal{P}\left(  \phi\right)  \left(  \lambda\right)  \right\Vert
_{\mathcal{HS}}^{2}  &  =\left\Vert \mathbf{u}_{\lambda}\otimes\mathbf{e}%
_{\lambda}\left\vert \det B\left(  \lambda\right)  \right\vert ^{-1/2}%
\right\Vert _{\mathcal{HS}}^{2}\\
&  =\left\vert \det B\left(  \lambda\right)  \right\vert ^{-1}\left\Vert
\mathbf{u}_{\lambda}\right\Vert _{L^{2}\left(
\mathbb{R}
^{d}\right)  }^{2}\\
&  =1.
\end{align*}
Since $N$ is unimodular, and since $\mu\left(  \mathbf{E}^{\circ}\right)
<\infty$ then from \cite{Fuhr cont}, Page $126,$ $\left(  L,\mathbf{H}\left(
\mathbf{e,E}^{\circ}\right)  \right)  $ is an admissible representation of $N$
and $\phi$ is an admissible vector for the representation $\left(
L,\mathbf{H}\left(  \mathbf{e,E}^{\circ}\right)  \right)  $.
\end{proof}

\begin{remark}
A proof of Theorem \ref{steps} is derived by applying Proposition \ref{sinc1}, Lemma \ref{ONB} and Lemma \ref{admissible}.
\end{remark}

\subsection{Additional Observations}

Let $m$ be the Lebesgue measure on $%
\mathbb{R}
^{n-2d}.$ Given two measurable sets $A,B\subseteq%
\mathbb{R}
^{n-2d},$
\[
A\Delta B=\left(  A-B\right)  \cup\left(  B-A\right)
\]
is the symmetric difference of the sets. Now, let us assume that there exists
a fundamental domain $\Lambda$ of $%
\mathbb{Z}
^{n-2d}$ such that%
\[
m\left(  \mathbf{E}^{\circ}\Delta\left(
{\displaystyle\bigcup\limits_{j\in S}}
\left(  \Lambda+k_{j}\right)  \right)  \right)  =0\mathbf{.}%
\]
That is, up to a set of Lebesgue measure zero, $\mathbf{E}^{\circ}$ is a
finite disjoint union of sets which are $%
\mathbb{Z}
^{n-2d}$-congruent to a fundamental domain of $%
\mathbb{R}
^{n-2d}.$ We acknowledge that this is a very strong condition to impose.
However, under this condition, we would like to present some simple sufficient
conditions for the statement
\[
\left(  \left\langle \mathbf{f},\pi_{\lambda-\kappa_{j}}\left(  \gamma
_{1}\right)  \mathbf{u}_{\lambda-\kappa_{j}}\right\rangle \right)
_{\gamma_{1}\in\Gamma_{\mathfrak{b}}\Gamma_{\mathfrak{a}}}\perp\left(
\left\langle \mathbf{g},\pi_{\lambda-\kappa_{j^{\prime}}}\left(  \gamma
_{1}\right)  \mathbf{u}_{\lambda-\kappa_{j^{\prime}}}\right\rangle \right)
_{\gamma_{1}\in\Gamma_{\mathfrak{b}}\Gamma_{\mathfrak{a}}}%
\]
given in Lemma \ref{ONB}.

\begin{lemma}
\label{Bl}Let us assume that there exists a fundamental domain $\Lambda$ of $%
\mathbb{Z}
^{n-2d}$ such that
\[
m\left(  \mathbf{E}^{\circ}\Delta\left(
{\displaystyle\bigcup\limits_{j\in S}}
\left(  \Lambda+k_{j}\right)  \right)  \right)  =0.
\]
For $j\neq j^{\prime},\lambda\in\Lambda$ and $\mathbf{u}_{\lambda-\kappa_{j}%
},\mathbf{u}_{\lambda-\kappa_{j^{\prime}}}\in L^{2}\left(
\mathbb{R}
^{d}\right)  $ as given in Lemma \ref{ONB}, if for any fixed $m\in%
\mathbb{Z}
^{d},$%
\[%
{\displaystyle\bigcup\limits_{\kappa_{s}\in S}}
B\left(  \lambda-\kappa_{s}\right)  ^{tr}\left(  E\left(  \lambda-\kappa
_{s}\right)  +m\right)  \text{ is a subset of a fundamental domain for }%
\mathbb{Z}
^{d}%
\]
and if%
\[
B\left(  \lambda-\kappa_{j}\right)  ^{tr}\left(  E\left(  \lambda-\kappa
_{j}\right)  +m\right)  \cap B\left(  \lambda-\kappa_{j^{\prime}}\right)
^{tr}\left(  E\left(  \lambda-\kappa_{j^{\prime}}\right)  +m\right)
\]
is a null set then
\[
\left(  \left\langle \mathbf{f},\pi_{\lambda-\kappa_{j}}\left(  \gamma
_{1}\right)  \mathbf{u}_{\lambda-\kappa_{j}}\right\rangle \right)
_{\gamma_{1}\in\Gamma_{\mathfrak{b}}\Gamma_{\mathfrak{a}}}\perp\left(
\left\langle \mathbf{g},\pi_{\lambda-\kappa_{j^{\prime}}}\left(  \gamma
_{1}\right)  \mathbf{u}_{\lambda-\kappa_{j^{\prime}}}\right\rangle \right)
_{\gamma_{1}\in\Gamma_{\mathfrak{b}}\Gamma_{\mathfrak{a}}}%
\]
for all $\mathbf{f},\mathbf{g}\in L^{2}\left(
\mathbb{R}
^{d}\right)  .$
\end{lemma}

\begin{proof}
Clearly, in order to compute the inner product of the sequences
\[
\left(  \left\langle \mathbf{f},\pi_{\lambda-\kappa_{j}}\left(  \gamma
_{1}\right)  \mathbf{u}_{\lambda-\kappa_{j}}\right\rangle \right)
_{\gamma_{1}\in\Gamma_{\mathfrak{b}}\Gamma_{\mathfrak{a}}},\left(
\left\langle \mathbf{g},\pi_{\lambda-\kappa_{j^{\prime}}}\left(  \gamma
_{1}\right)  \mathbf{u}_{\lambda-\kappa_{j^{\prime}}}\right\rangle \right)
_{\gamma_{1}\in\Gamma_{\mathfrak{b}}\Gamma_{\mathfrak{a}}}\in l^{2}\left(
\Gamma_{\mathfrak{b}}\Gamma_{\mathfrak{a}}\right)
\]
we need to calculate a formula for the following sum:
\[
\sum_{\gamma_{1}\in\Gamma_{\mathfrak{b}}\Gamma_{\mathfrak{a}}}\left(
\left\langle \mathbf{f},\pi_{\lambda-\kappa_{j}}\left(  \gamma_{1}\right)
\mathbf{u}_{\lambda-\kappa_{j}}\right\rangle \overline{\left\langle
\mathbf{g},\pi_{\lambda-\kappa_{j^{\prime}}}\left(  \gamma_{1}\right)
\mathbf{u}_{\lambda-\kappa_{j^{\prime}}}\right\rangle }\right)  .
\]
\ First,
\begin{align*}
\left\langle \mathbf{f},\pi_{\lambda-\kappa_{j}}\left(  \gamma_{1}\right)
\mathbf{u}_{\lambda-\kappa_{j}}\right\rangle  &  =\int_{\mathbf{%
\mathbb{R}
}}\mathbf{f}\left(  t\right)  \overline{\pi_{\lambda-\kappa_{j}}\left(
\gamma_{1}\right)  \mathbf{u}_{\lambda-\kappa_{j}}\left(  t\right)  }dt\\
&  =\int_{\mathbf{%
\mathbb{R}
}}\mathbf{f}\left(  t\right)  \overline{e^{2\pi i\left\langle l,B\left(
\lambda-\kappa_{j}\right)  ^{tr}t\right\rangle }\mathbf{u}_{\lambda-\kappa
_{j}}\left(  t-m\right)  }dt.
\end{align*}
Put $s=B\left(  \lambda-\kappa_{j}\right)  ^{tr}t.$ We recall that
$\mathbf{u}_{\lambda}=\left\vert \det B\left(  \lambda\right)  \right\vert
^{1/2}\chi_{E\left(  \lambda\right)  }.$ So,
\begin{align*}
&  \left\langle \mathbf{f},\pi_{\lambda-\kappa_{j}}\left(  \gamma_{1}\right)
\mathbf{u}_{\lambda-\kappa_{j}}\right\rangle \\
&  =\int_{B\left(  \lambda-\kappa_{j}\right)  ^{tr}\left(  E\left(
\lambda-\kappa_{j}\right)  +m\right)  }\frac{\mathbf{f}\left(  B\left(
\lambda-\kappa_{j}\right)  ^{-tr}s\right)  }{\left\vert \det B\left(
\lambda-\kappa_{j}\right)  \right\vert ^{1/2}}e^{-2\pi i\left\langle
l,s\right\rangle }ds.
\end{align*}
Similarly,
\begin{align*}
&  \left\langle \mathbf{g},\pi_{\lambda-\kappa_{j^{\prime}}}\left(  \gamma
_{1}\right)  \mathbf{u}_{\lambda-\kappa_{j^{\prime}}}\right\rangle \\
&  =\int_{B\left(  \lambda-\kappa_{j^{\prime}}\right)  ^{tr}\left(  E\left(
\lambda-\kappa_{j^{\prime}}\right)  +m\right)  }\frac{\mathbf{g}\left(
B\left(  \lambda-\kappa_{j^{\prime}}\right)  ^{-tr}s\right)  }{\left\vert \det
B\left(  \lambda-\kappa_{j^{\prime}}\right)  \right\vert ^{1/2}}e^{-2\pi
i\left\langle l,s\right\rangle }ds.
\end{align*}
If for $\lambda-\kappa_{j}\in A_{j},$
\[%
{\displaystyle\bigcup\limits_{\kappa_{s}\in S}}
B\left(  \lambda-\kappa_{s}\right)  ^{tr}\left(  E\left(  \lambda-\kappa
_{s}\right)  +m\right)
\]
is a subset of a fundamental domain of $%
\mathbb{Z}
^{d}$ for distinct $\kappa_{j},\kappa_{j^{\prime}}\in S,$ and if
\[
B\left(  \lambda-\kappa_{j}\right)  ^{tr}\left(  E\left(  \lambda-\kappa
_{j}\right)  +m\right)  \cap B\left(  \lambda-\kappa_{j^{\prime}}\right)
^{tr}\left(  E\left(  \lambda-\kappa_{j^{\prime}}\right)  +m\right)
\]
is a null set then
\[
\left(  \left\langle \mathbf{f},\pi_{\lambda-\kappa_{j}}\left(  \gamma
_{1}\right)  \mathbf{u}_{\lambda-\kappa_{j}}\right\rangle \right)
_{\gamma_{1}\in\Gamma_{\mathfrak{b}}\Gamma_{\mathfrak{a}}}\perp\left(
\left\langle \mathbf{g},\pi_{\lambda-\kappa_{j^{\prime}}}\left(  \gamma
_{1}\right)  \mathbf{u}_{\lambda-\kappa_{j^{\prime}}}\right\rangle \right)
_{\gamma_{1}\in\Gamma_{\mathfrak{b}}\Gamma_{\mathfrak{a}}};
\]
because
\[
\left(  \left\langle \mathbf{f},\pi_{\lambda-\kappa_{j}}\left(  \gamma
_{1}\right)  \mathbf{u}_{\lambda-\kappa_{j}}\right\rangle \right)
_{\gamma_{1}\in\Gamma_{\mathfrak{b}}\Gamma_{\mathfrak{a}}},\text{ and }\left(
\left\langle \mathbf{g},\pi_{\lambda-\kappa_{j^{\prime}}}\left(  \gamma
_{1}\right)  \mathbf{u}_{\lambda-\kappa_{j^{\prime}}}\right\rangle \right)
_{\gamma_{1}\in\Gamma_{\mathfrak{b}}\Gamma_{\mathfrak{a}}}%
\]
are Fourier inverses of the following orthogonal functions
\begin{align}
\theta_{\mathbf{f,}\lambda,j,m}\left(  s\right)   &  =\chi_{B\left(
\lambda-\kappa_{j}\right)  ^{tr}\left(  E\left(  \lambda-\kappa_{j}\right)
+m\right)  }\left(  s\right)  \frac{\mathbf{f}\left(  B\left(  \lambda
-\kappa_{j}\right)  ^{-tr}s\right)  }{\left\vert \det B\left(  \lambda
-\kappa_{j}\right)  \right\vert ^{1/2}},\text{ and }\label{ort1}\\
\theta_{\mathbf{g,}\lambda,j^{\prime},m}\left(  s\right)   &  =\chi_{B\left(
\lambda-\kappa_{j^{\prime}}\right)  ^{tr}\left(  E\left(  \lambda
-\kappa_{j^{\prime}}\right)  +m\right)  }\left(  s\right)  \frac
{\mathbf{g}\left(  B\left(  \lambda-\kappa_{j^{\prime}}\right)  ^{-tr}%
s\right)  }{\left\vert \det B\left(  \lambda-\kappa_{j^{\prime}}\right)
\right\vert ^{1/2}}\text{ respectively.} \label{orth2}%
\end{align}
In fact, we think of the functions above (\ref{ort1}), (\ref{orth2}) as being
elements of $L^{2}\left(  \mathbf{I}_{m,\lambda}\right)  $ such that
$\mathbf{I}_{m,\lambda}$ is a fundamental domain for $%
\mathbb{Z}
^{d}.$ Combining the observations made above, we obtain that for any
$\mathbf{f},\mathbf{g}\in L^{2}\left(
\mathbb{R}
^{d}\right)  ,$
\begin{align*}
&  \sum_{\gamma_{1}\in\Gamma_{\mathfrak{b}}\Gamma_{\mathfrak{a}}}\left\langle
\mathbf{f},\pi_{\lambda-\kappa_{j}}\left(  \gamma_{1}\right)  \mathbf{u}%
_{\lambda-\kappa_{j}}\right\rangle \overline{\left\langle \mathbf{g}%
,\pi_{\lambda-\kappa_{j^{\prime}}}\left(  \gamma_{1}\right)  \mathbf{u}%
_{\lambda-\kappa_{j^{\prime}}}\right\rangle }\\
&  =\sum_{m\in%
\mathbb{Z}
^{d}}\sum_{l\in%
\mathbb{Z}
^{d}}\left(  \int_{\mathbf{I}_{m,\lambda}}\theta_{\mathbf{f,}\lambda
,j,m}\left(  s\right)  e^{-2\pi i\left\langle l,s\right\rangle }ds\right)
\overline{\left(  \int_{\mathbf{I}_{m,\lambda}}\theta_{\mathbf{g,}%
\lambda,j^{\prime},m}\left(  s\right)  e^{-2\pi i\left\langle l,x\right\rangle
}dx\right)  }\\
&  =\sum_{m\in%
\mathbb{Z}
^{d}}\sum_{l\in%
\mathbb{Z}
^{d}}\widehat{\theta_{\mathbf{f,}\lambda,j,m}}\left(  l\right)  \overline
{\widehat{\theta_{\mathbf{g,}\lambda,j^{\prime},m}}\left(  l\right)  }\\
&  =\sum_{m\in%
\mathbb{Z}
^{d}}\overset{=0}{\overbrace{\sum_{l\in%
\mathbb{Z}
^{d}}\widehat{\theta_{\mathbf{f,}\lambda,j,m}}\left(  l\right)  \overline
{\widehat{\theta_{\mathbf{g,}\lambda,j^{\prime},m}}\left(  l\right)  }}}\\
&  =0.
\end{align*}
Thus
\[
\left(  \left\langle \mathbf{f},\pi_{\lambda-\kappa_{j}}\left(  \gamma
_{1}\right)  \mathbf{u}_{\lambda-\kappa_{j}}\right\rangle \right)
_{\gamma_{1}\in\Gamma_{\mathfrak{b}}\Gamma_{\mathfrak{a}}}\perp\left(
\left\langle \mathbf{g},\pi_{\lambda-\kappa_{j^{\prime}}}\left(  \gamma
_{1}\right)  \mathbf{u}_{\lambda-\kappa_{j^{\prime}}}\right\rangle \right)
_{\gamma_{1}\in\Gamma_{\mathfrak{b}}\Gamma_{\mathfrak{a}}}.
\]
This concludes the proof.
\end{proof}

In light of Lemma \ref{ONB} and Lemma \ref{Bl}, the following holds true.

\begin{proposition}
\label{onb}Let us assume that there exists a fundamental domain $\Lambda$ of $%
\mathbb{Z}
^{n-2d}$ such that
\[
m\left(  \mathbf{E}^{\circ}\Delta\left(
{\displaystyle\bigcup\limits_{j\in S}}
\left(  \Lambda+k_{j}\right)  \right)  \right)  =0.
\]
If $%
{\displaystyle\bigcup\limits_{\kappa_{s}\in S}}
B\left(  \lambda-\kappa_{s}\right)  ^{tr}\left(  E\left(  \lambda-\kappa
_{s}\right)  +m\right)  $ is a subset of a fundamental domain of $%
\mathbb{Z}
^{d}$ and if
\[
B\left(  \lambda-\kappa_{j}\right)  ^{tr}\left(  E\left(  \lambda-\kappa
_{j}\right)  +m\right)  \cap B\left(  \lambda-\kappa_{j^{\prime}}\right)
^{tr}\left(  E\left(  \lambda-\kappa_{j^{\prime}}\right)  +m\right)
\]
is a null set for $\lambda\in\Lambda\ $for $m\in%
\mathbb{Z}
^{d}$ and for distinct $\kappa_{j},\kappa_{j^{\prime}}\in S$ then $L\left(
\Gamma\right)  \left(  \phi\right)  $ is an orthonormal basis for the Hilbert
space $\mathbf{H}\left(  \mathbf{e,E}^{\circ}\right)  $.
\end{proposition}

\subsection{Examples}

\begin{example}
Let $N$ be a nilpotent Lie group with Lie algebra spanned by the vectors
$Z_{1},Z_{2},Y_{1},Y_{2},X_{1},X_{2}$ with the following non-trivial Lie
brackets:
\[
\left[  X_{1},Y_{1}\right]  =Z_{1},\left[  X_{2},Y_{2}\right]  =Z_{2}.
\]
In this example, the discrete set
\[
\Gamma=\exp\left(
\mathbb{Z}
Z_{1}+%
\mathbb{Z}
Z_{2}\right)  \exp\left(
\mathbb{Z}
Y_{1}+%
\mathbb{Z}
Y_{2}\right)  \exp\left(
\mathbb{Z}
X_{1}+%
\mathbb{Z}
X_{2}\right)
\]
is actually a uniform subgroup of the Lie group $N.$ Moreover, $N$ is a direct
product of two Heisenberg groups and satisfies all properties described in
Condition \ref{conditiion}. Next, we will apply Theorem \ref{steps} to show
that there exists a left-invariant subspace of $L^{2}\left(  N\right)  $ which
is a $\Gamma$-sampling space with the interpolation property. First, it is
easy to check that
\[
\mu\left(  \mathbf{E}\right)  =\mu\left(  \left\{  \left(  \lambda_{1}%
,\lambda_{2}\right)  \in\mathbb{R}^{2}:\lambda_{1}\lambda_{2}\neq0\text{ and
}\left\vert \lambda_{1}\lambda_{2}\right\vert \leq1\right\}  \right)
=\infty.
\]
Now, let $\mathbf{R}=\left[  -1,1\right]  ^{2}.$ Then
\[
\mu\left(  \mathbf{E\cap R}\right)  =\int_{-1}^{1}\int_{-1}^{1}\left\vert
\lambda_{1}\lambda_{2}\right\vert d\lambda_{1}d\lambda_{2}=1.
\]
Next, we observe for each $\lambda\in\mathbf{E\cap R,}$ $\left[  0,1\right)
^{2}$ tiles $\mathbb{R}^{2}$ by $\mathbb{Z}^{2}$ and packs $\mathbb{R}^{2}$
by
\[
B\left(  \lambda\right)  ^{-tr}\mathbb{Z}^{2}=\left[
\begin{array}
[c]{cc}%
\lambda_{1} & 0\\
0 & \lambda_{2}%
\end{array}
\right]  ^{-1}%
\mathbb{Z}
^{2}.
\]
Next, it is not too hard to check that%
\[
\mathbf{E}^{\circ}\text{ }\mathbf{=}\left\{  \left(  \lambda_{1},\lambda
_{2}\right)  \in\left[  -1,1\right]  ^{2}:\lambda_{1}\lambda_{2}\neq0\right\}
,
\]
and
\[
m\left(  \mathbf{E}^{\circ}\Delta\left(
{\displaystyle\bigcup\limits_{j\in S}}
\left(  \left[  0,1\right)  ^{2}+j\right)  \right)  \right)  =0
\]
where%
\[
S=\left\{  \left[
\begin{array}
[c]{c}%
0\\
0
\end{array}
\right]  ,\left[
\begin{array}
[c]{c}%
-1\\
0
\end{array}
\right]  ,\left[
\begin{array}
[c]{c}%
-1\\
-1
\end{array}
\right]  ,\left[
\begin{array}
[c]{c}%
0\\
-1
\end{array}
\right]  \right\}  .
\]
Moreover, for $\lambda\in\left[  0,1\right)  ^{2}$ and for $m\in%
\mathbb{Z}
^{2},$ we define
\begin{align*}
M_{\lambda,1}  &  =\left[
\begin{array}
[c]{cc}%
\lambda_{1} & 0\\
0 & \lambda_{2}%
\end{array}
\right]  ,M_{\lambda,2}=\left[
\begin{array}
[c]{cc}%
\lambda_{1}-1 & 0\\
0 & \lambda_{2}-1
\end{array}
\right]  ,\\
M_{\lambda,3}  &  =\left[
\begin{array}
[c]{cc}%
\lambda_{1}-1 & 0\\
0 & \lambda_{2}%
\end{array}
\right]  ,M_{\lambda,4}=\left[
\begin{array}
[c]{cc}%
\lambda_{1} & 0\\
0 & \lambda_{2}-1
\end{array}
\right]  ,\\
\mathbf{S}_{\lambda,1,j}  &  =M_{\lambda,1}\left(  \left(  0,1\right)
^{2}+j\right)  ,\text{ }\mathbf{S}_{\lambda,2,j}=M_{\lambda,2}\left(  \left(
0,1\right)  ^{2}+j\right) \\
\mathbf{S}_{\lambda,3,j}  &  =M_{\lambda,3}\left(  \left(  0,1\right)
^{2}+j\right)  ,\mathbf{S}_{\lambda,4,m}=M_{\lambda,4}\left(  \left(
0,1\right)  ^{2}+j\right)  .
\end{align*}
We observe that for all $j\in%
\mathbb{Z}
^{2},$
\[%
{\displaystyle\bigcup\limits_{k=1}^{4}}
\mathbf{S}_{\lambda,k,j}%
\]
is a subset of a fundamental domain for $%
\mathbb{Z}
^{2}$ and $\mathbf{S}_{\lambda,\ell_{1},j}\cap\mathbf{S}_{\lambda,\ell_{2},j}$
is a null set for distinct $\ell_{1},\ell_{2}$ and for almost every
$\lambda\in\left(  0,1\right)  ^{2}.$ Let $\phi\in\mathbf{H}\left(
\mathbf{e,E}^{\circ}\right)  ,$ such that
\[
\mathcal{P}\phi\left(  \lambda\right)  =\chi_{\left[  0,1\right)  ^{2}}\left(
t\right)  \otimes\chi_{\left[  0,1\right)  ^{2}}\left(  t\right)  .
\]
According to Theorem \ref{steps} and Proposition \ref{onb} $L\left(
\Gamma\right)  \phi$ is an orthonormal basis for $\mathbf{H}\left(
\mathbf{e,E}^{\circ}\right)  $ and $W_{\phi}\left(  \mathbf{H}\left(
\mathbf{e,E}^{\circ}\right)  \right)  $ is a $\Gamma$-sampling space with the
interpolation property.
\end{example}

\begin{example}
Put $\alpha=\left(  \frac{\pi}{2}\right)  ^{1/4}.$ Let $N$ be a nilpotent Lie
group with Lie algebra spanned by the vectors $Z_{1},Z_{2},Y_{1},Y_{2}%
,X_{1},X_{2}$ with the following non-trivial Lie brackets:
\begin{align*}
\left[  X_{1},Y_{1}\right]   &  =\alpha Z_{1},\left[  X_{1},Y_{2}\right]
=-\alpha Z_{2}\\
\left[  X_{2},Y_{1}\right]   &  =\alpha Z_{2},\left[  X_{2},Y_{2}\right]
=\alpha Z_{1}.
\end{align*}
Here
\[
B\left(  \lambda\right)  =\left[
\begin{array}
[c]{cc}%
\alpha\lambda_{1} & -\alpha\lambda_{2}\\
\alpha\lambda_{2} & \alpha\lambda_{1}%
\end{array}
\right]  ,
\]
and
\[
\Gamma=\exp\left(
\mathbb{Z}
Z_{1}+%
\mathbb{Z}
Z_{2}\right)  \exp\left(
\mathbb{Z}
Y_{1}+%
\mathbb{Z}
Y_{2}\right)  \exp\left(
\mathbb{Z}
X_{1}+%
\mathbb{Z}
X_{2}\right)  .
\]
Next, the Plancherel measure is
\[
\alpha^{2}\left(  \lambda_{1}^{2}+\lambda_{2}^{2}\right)  d\lambda_{1}%
d\lambda_{2}%
\]
and is supported on the manifold
\[
\Sigma=\left\{  \left(  \lambda_{1},\lambda_{2}\right)  \in\mathbb{R}%
^{2}:\left(  \lambda_{1},\lambda_{2}\right)  \neq\left(  0,0\right)  \right\}
.
\]
Now, we have%
\[
\mathbf{E}=\left\{  \left(  \lambda_{1},\lambda_{2}\right)  \in\Sigma:\text{
}\lambda_{1}^{2}+\lambda_{2}^{2}\leq\frac{1}{\alpha^{2}}\right\}
\]
and
\[
\mu\left(  \mathbf{E}\right)  =\int_{0}^{2\pi}\int_{0}^{\left(  \frac{\pi}%
{2}\right)  ^{-1/4}}r^{3}drd\theta=1.
\]
Thus, in this example $\mathbf{E}^{\circ}=\mathbf{E.}$ Next, we partition the
set $\mathbf{E}^{\circ}$ such that
\[
\mathbf{E}^{\circ}=%
{\displaystyle\bigcup\limits_{j\in S}}
\left(  \left(  \left[  0,1\right)  ^{2}+j\right)  \cap\mathbf{E}^{\circ
}\right)
\]
where%
\[
S=\left\{  \left[
\begin{array}
[c]{c}%
0\\
0
\end{array}
\right]  ,\left[
\begin{array}
[c]{c}%
-1\\
0
\end{array}
\right]  ,\left[
\begin{array}
[c]{c}%
-1\\
-1
\end{array}
\right]  ,\left[
\begin{array}
[c]{c}%
0\\
-1
\end{array}
\right]  \right\}  .
\]
For each $j\in S,$ put $A_{j}=\left(  \left[  0,1\right)  ^{2}\cap
\mathbf{E}^{\circ}\right)  +j.$ Next, appealing to Theorem \ref{main}, for
each $j\in S,$ there exists $\phi_{j}\in\mathbf{H}\left(  \mathbf{e,}%
A_{j}\right)  $ such that $W_{\phi_{j}}\left(  \mathbf{H}\left(
\mathbf{e,}A_{j}\right)  \right)  $ is a $\Gamma$-sampling subspace of
$L^{2}\left(  N\right)  $ with sinc-type function $W_{\phi_{j}}(\phi_{j}).$
Also for each $j\in S,$ $W_{\phi_{j}}\left(  \mathbf{H}\left(  \mathbf{e,}%
A_{j}\right)  \right)  $ does not have the interpolation property with respect
to $\Gamma$ since $\mu\left(  A_{j}\right)  <1$ for each $j\in S.$ Although
\[
\left\Vert {\sum\limits_{j\in S}}\phi_{j}\right\Vert ^{2}=1,
\]
we cannot say that $L\left(  \Gamma\right)  \left(  {\sum\limits_{j\in S}}%
\phi_{j}\right)  $ is a Parseval frame in
\[
\mathbf{H}\left(  \mathbf{e,E}^{\circ}\right)  =%
{\displaystyle\bigoplus\limits_{j\in S}}
\mathbf{H}\left(  \mathbf{e,}A_{j}\right)  .
\]
Suppose that we define the map $\lambda\mapsto\mathbf{u}_{\lambda}$ on
$\Sigma$ such that
\[
\mathbf{u}_{\lambda}=\left\{
\begin{array}
[c]{c}%
\left\vert \det B\left(  \lambda\right)  \right\vert ^{1/2}\chi_{E\left(
\lambda\right)  }\text{ if }\lambda\in\mathbf{E}^{\circ}\\
0\text{ if }\lambda\notin\Sigma-\mathbf{E}^{\circ}%
\end{array}
\right.
\]
where $E\left(  \lambda\right)  \subset%
\mathbb{R}
^{2}$ tiles $\mathbb{R}^{2}$ by $\mathbb{Z}^{2}$ and packs $\mathbb{R}^{2}$
by
\[
B\left(  \lambda\right)  ^{-tr}\mathbb{Z}^{2}=\left[
\begin{array}
[c]{cc}%
\dfrac{\lambda_{1}}{\alpha\lambda_{1}^{2}+\alpha\lambda_{2}^{2}} &
-\dfrac{\lambda_{2}}{\alpha\lambda_{1}^{2}+\alpha\lambda_{2}^{2}}\\
\dfrac{\lambda_{2}}{\alpha\lambda_{1}^{2}+\alpha\lambda_{2}^{2}} &
\dfrac{\lambda_{1}}{\alpha\lambda_{1}^{2}+\alpha\lambda_{2}^{2}}%
\end{array}
\right]  \mathbb{Z}^{2}.
\]
Finally, we define $\phi\in L^{2}\left(  N\right)  $ such that
\[
\mathcal{P}\phi\left(  \lambda\right)  =\frac{\mathbf{u}_{\lambda}%
\otimes\mathbf{e}_{\lambda}}{\sqrt{\left\vert \det B\left(  \lambda\right)
\right\vert }}.
\]
According to Theorem \ref{steps}, if for each $1\leq j,j^{\prime}%
\leq\mathrm{card}\left(  S\right)  ,$ $j\neq j^{\prime},$ for $\lambda
\in\left[  0,1\right)  ^{2},$ for arbitrary functions $\mathbf{f}%
,\mathbf{g}\in L^{2}\left(  \mathbb{R}^{2}\right)  ,$ and for distinct
$j,j^{\prime}\in S$
\[
\left(  \left\langle \mathbf{f},\pi_{\lambda-j}\left(  \gamma_{1}\right)
\mathbf{u}_{\lambda-j}\right\rangle \right)  _{\gamma_{1}\in\Gamma
_{\mathfrak{b}}\Gamma_{\mathfrak{a}}}\perp\left(  \left\langle \mathbf{g}%
,\pi_{\lambda-j^{\prime}}\left(  \gamma_{1}\right)  \mathbf{u}_{\lambda
-j^{\prime}}\right\rangle \right)  _{\gamma_{1}\in\Gamma_{\mathfrak{b}}%
\Gamma_{\mathfrak{a}}}\text{,}%
\]
then $W_{\phi}\left(  \mathbf{H}\left(  \mathbf{e,E}^{\circ}\right)  \right)
$ is a $\Gamma$-sampling space with sinc-type function $W_{\phi}\left(
\phi\right)  .$ Moreover, $W_{\phi}\left(  \mathbf{H}\left(  \mathbf{e,E}%
^{\circ}\right)  \right)  $ has the interpolation property as well.
\end{example}

\begin{acknowledgement}
Many thanks go to the anonymous referee for a very thorough reading of this
paper. His suggestions, remarks and corrections greatly improved the quality
of the work. I also dedicate this paper to the loving memory of my mother Olga.
\end{acknowledgement}


\begin{thebibliography}{9}                                                                                                %


\bibitem {Pete}P.G. Casazza, The Art of Frame Theory, Taiwanese Journal of
Math, Vol 4 (2) (2000) 129-202

\bibitem {Corwin}L. Corwin, F. Greenleaf, Representations of Nilpotent Lie
Groups and their Applications. Part I. Basic Theory and Examples, Cambridge
Studies in Advanced Mathematics, 18. Cambridge University Press, Cambridge, (1990)

\bibitem {Currey}B. Currey, A. Mayeli, A Density Condition for Interpolation
on the Heisenberg Group, Rocky Mountain J. Math. Volume 42, Number 4 (2012), 1135-1151

\bibitem {Folland}G. Folland, A Course in Abstract Harmonic Analysis, Studies
in Advanced Mathematics. CRC Press, Boca Raton, FL, 1995

\bibitem {Fuhr cont}H. F\"{u}hr, Abstract Harmonic Analysis of Continuous
Wavelet Transforms, Springer Lecture Notes in Math. 1863, (2005).

\bibitem {Han Yang Wang}D. Han and Y. Wang, Lattice Tiling and the Weyl
Heisenberg Frames, Geom. Funct. Anal. 11 (2001), no. 4, 742--758

\bibitem {oussa1}V. Oussa, Sinc Type Functions on a Class of nilpotent Lie
groups, to appear in Advances in Pure and Applied Mathematics, (2014)

\bibitem {oussa}V. Oussa, Bandlimited Spaces on Some 2-step Nilpotent Lie
Groups With One Parseval Frame Generator, to appear in Rocky Mountain Journal
of Mathematics

\bibitem {Pfander}G. Pfander, P. Rashkov, Y. Wang, A Geometric Construction of
Tight Multivariate Gabor Frames with Compactly Supported Smooth Windows, J.
Fourier Anal. Appl. 18 (2012), no. 2, 223--239. 42C15
\end{thebibliography}
\end{document}